\documentclass[a4paper,DIV=11,abstracton]{scrartcl}
\usepackage{fullpage,authblk,amsmath,amssymb,amsthm,amsfonts}
\usepackage{bm, bbm, dsfont}
\usepackage{enumerate}
\usepackage{tikz,algorithm2e}
\usepackage{caption}
\usepackage{subcaption}
\usepackage{url,mathtools}
\usepackage{footmisc}

\usepackage[pdffitwindow=true,
		   pdfstartview={FitH},
            plainpages=false,
            pdfpagelabels=true,
            pdfpagemode=UseOutlines,
            pdfpagelayout=SinglePage,
            bookmarks=false,
            colorlinks=true,
            hyperfootnotes=false,
            linkcolor=blue,
            urlcolor=blue!30!black,
            citecolor=green!50!black]{hyperref}
            
\usepackage{color}
\usepackage[colorinlistoftodos,prependcaption,textsize=tiny]{todonotes}

\theoremstyle{plain}
\newtheorem{theorem}{Theorem}[section]
\newtheorem{lemma}[theorem]{Lemma}
\newtheorem{proposition}[theorem]{Proposition}
\newtheorem{corollary}[theorem]{Corollary}
\newtheorem{assumption}[theorem]{Assumption}
\newtheorem{remark}[theorem]{Remark}

\theoremstyle{definition}
\newtheorem{definition}[theorem]{Definition}

\numberwithin{equation}{section}

\newcommand{\card}[1]{\textup{card}({#1})}
\newcommand{\inter}[1]{\textup{int}({#1})}
\newcommand{\clos}[1]{\textup{cl}({#1})}
\newcommand{\diam}[1]{\textup{diam}({#1})}
\newcommand{\argmax}{\textup{argmax}}
\newcommand{\Lip}{\textup{Lip}}

\newcommand{\abs}[1]{\left\vert#1\right\vert}
\newcommand{\Vrt}[1]{\left\Vert #1 \right\Vert}

\newcommand{\bb}[1]{\mathbb{#1}}

\begin{document}
\title{A convergent discretisation method for transition path theory for diffusion processes}

\author{Nada Cvetkovi\'{c}\thanks{\ Corresponding author: nada.cvetkovic@uni-potsdam.de} }
\author{Han Cheng Lie\thanks{\ hanlie@uni-potsdam.de}}
\affil{Universit\"{a}t Potsdam, Institut f\"{u}r Mathematik, Karl-Liebknecht Str. 24/25, D-14476 Potsdam OT Golm, Germany}

\author{Tim Conrad\thanks{\ conrad@mi.fu-berlin.de}}
\affil{Department of Mathematics and Computer Science, Freie Universit\"at Berlin, Arnimallee 6, 14195 Berlin, Germany}
\renewcommand\Affilfont{\small}

\date{}

\maketitle

\begin{abstract}
	Transition path theory (TPT) for diffusion processes is a framework for analysing the transitions of multiscale ergodic diffusion processes between disjoint metastable subsets of state space. Most methods for applying TPT involve the construction of a Markov state model on a discretisation of state space that approximates the underlying diffusion process. However, the assumption of Markovianity is difficult to verify in practice, and there are to date no known error bounds or convergence results for these methods. We propose a Monte Carlo method for approximating the forward committor, probability current, and streamlines from TPT for diffusion processes. Our method uses only sample trajectory data and partitions of state space based on Voronoi tessellations. It does not require the construction of a Markovian approximating process. We rigorously prove error bounds for the approximate TPT objects and use these bounds to show convergence to their exact counterparts in the limit of arbitrarily fine discretisation. We illustrate some features of our method by application to a process that solves the Smoluchowski equation on a triple-well potential.
\end{abstract}

\section{Introduction}
\label{sec:intro}

In many applications, one is often interested in understanding the rare transitions of a stochastic dynamical system between two metastable subsets of state space. Here, `metastable set' refers to a set in which the system spends a long time on average before exiting, and `rare' means that the probability of observing a transition between any two metastable sets is small. The multiscale behaviour of such a system is captured by the fact that the system spends very short times outside the union of metastable sets and spends long times inside this union. 

Transition path theory (TPT) is a framework for the analysis of transition events of multiscale ergodic diffusion processes between two metastable subsets of state space \cite{E2005242,weinan2006towards}. TPT has played an important role in the research area of molecular dynamics, where the conformational changes of a molecule occur over much longer time scales (milliseconds or longer) and fluctuations in bond lengths occur over very short time scales (nano- or femtoseconds). The conformational changes are of interest because they play an important role in chemical reactions, e.g. in cellular processes. A far from exhaustive list of references in this area of research is \cite{maragliano,NoePNAS,vandeneijnden_venturoli_ciccotti,bowman,silva_etal,Boekelheide,plattner,prinz,benjamin,Vanden-Eijnden2014,TPT_illustration}. In addition to molecular dynamics, TPT has been applied to drug design in pharmacology \cite{shukla2014activation}, to study flows in complex networks \cite{Cameron}, to study the temporal behavior of online topics on social networks \cite{lorenz2018tracking}, and to model atmospheric circulation in climate events \cite{ser2015dominant}. Versions of TPT for Markov jump processes and Markov chains \cite{TPT_MJP} and non-ergodic jump processes \cite{von2018statistical} have been developed. TPT has also been studied using the tools of stochastic analysis \cite{Lu2015}.

We now briefly outline the main concepts and objects of TPT for ergodic diffusion processes, and refer the reader to \cite{E2005242,weinan2006towards,TPT_illustration} for further details. For brevity, we will refer to ergodic diffusion processes simply as `processes' in what follows. Let $A$ and $B$ be two disjoint, simply connected, open subsets of the state space $S$ of a process $X$. Given an infinitely long trajectory of $X$, a finite segment of the trajectory is said to be `reactive' if, after leaving $A$, it enters $B$ before re-entering $A$. The probability that a trajectory of $X$ is reactive, conditioned on its current state being $X_0=x$, is given by the (forward) committor $q(x):= \mathbb{P}(X_{\tau_{A\cup B}(X)}\in B\vert X_0 = x)$. Here, $\tau_{A\cup B}(X)$ denotes the first hitting time of $X$ with respect to the set $A\cup B$. 

Under certain hypotheses on the drift and diffusion coefficient of the process, one can show that the committor is the unique solution of a Feynman-Kac or backward Kolmogorov partial differential equation (PDE) defined by the infinitesimal generator of the process and certain Dirichlet boundary conditions \cite[Eq. (10)]{weinan2006towards}. Consequently, the committor is twice continuously differentiable in the interior of $S\setminus (A\cup B)$. This exemplifies a general principle in probability theory, i.e. that objects that describe the averaged behaviour of the process exhibit certain regularity or smoothness properties. 

Using the above-mentioned principle, we can ask whether there exist smooth curves that represent the average behavior of reactive trajectories of the process, in the sense that for any $\epsilon>0$, the $L^\infty$ ball of radius $\epsilon$ centered at such a curve contains a reactive trajectory. In TPT these smooth curves are called `streamlines'; they are integral curves of a vector field called the `probability current'. Using the streamlines, one can construct so-called `transition tubes', by choosing a subset $A'$ on the boundary of $A$ and taking the union of the streamlines with initial conditions in $A'$. Given a probability measure on state space, e.g. the invariant Boltzmann-Gibbs measure, the probability that the process will transition from $A$ to $B$ via this transition tube can be computed from the measure of $A'\subset \partial A$ with respect to the surface measure on the boundary of $A$ that is induced by the measure on state space. The `dominant transition pathways' from $A$ to $B$ are the transition tubes with the highest probability. These pathways are important in applications, because they describe the most likely routes that will be taken by a multiscale stochastic dynamical system during transitions between metastable sets.

If the process is the solution to a Smoluchowski equation, i.e. a stochastic differential equation (SDE) where the drift coefficient is given by the gradient of a potential and the diffusion coefficient is proportional to the temperature, then one can additionally show that the probability current is proportional to the product of the Boltzmann-Gibbs density of the invariant measure with the gradient of the committor. If the state space is a subset of $\bb{R}^d$ for $d\in\{1,2,3\}$, then one can in principle completely avoid the task of collecting statistics of reactive trajectories, and apply deterministic numerical methods to compute the TPT objects described above. This can lead to useful visualisations of the behaviour of reactive trajectories, see e.g. \cite{TPT_illustration}. However, in many applications, the high dimension of the state space makes deterministic numerical methods for solving PDEs impractical. As a result, Monte Carlo-based methods are often used when computing TPT objects in high dimensions.

In many applications of TPT, one often uses a Markovian stochastic process to approximate the underlying diffusion process, where the state space of the Markovian approximation is obtained by discretising the state space of the diffusion process \cite{TPT_MJP}. The transition probabilities of the Markov approximation are estimated using short trajectory data. In molecular dynamics, these approximations are known as `Markov state models'; see e.g. \cite{SchuetteSarich} for a mathematical treatment of this subject and \cite{NoePNAS} for an application of TPT to a Markov state model. Since projecting the trajectories of the diffusion process onto the discrete state space results in a loss of the Markov property of the diffusion process, a correction step is needed to produce the Markov state model. The correction is often done by identifying a suitable `lag time', i.e. an observation time window which yields an approximately Markovian process on the discrete state space; see \cite{prinz2011markov} for an illustration. Although error analyses of Markov state models have been developed, e.g. in \cite{sarich2010approximation,SchuetteSarich}, the justification for the lag time selection is in practice heuristic, and the approximating process is non-Markovian. To the best of our knowledge, there is no error analysis for methods that use a non-Markovian process on a discrete state space to approximate TPT objects for a diffusion process.

\subsection{Contributions and outline}

The goal of this paper is to develop a new computational method for approximating certain objects of TPT for an ergodic diffusion process $X$. The inputs to this method are a discretisation of the state space $S$ of the underlying process $X$ and sample trajectories of $X$. The outputs are approximations of the committor, probability current, and streamlines of $X$. We develop error bounds for each approximation and use these error bounds to prove convergence of each approximate TPT object to the corresponding TPT object of $X$ in the `continuum limit' where the discretisation becomes arbitrarily uniformly fine.

We highlight some important features of our method. First, because our method computes the approximate TPT objects using only sample trajectories of $X$, it is not necessary to know or estimate the drift and diffusion coefficients of the SDE that defines the underlying process. It suffices to have a `black box' that generates the trajectory data. In this sense, our approach is `data-driven'. 

Second, our method does not involve constructing a Markovian stochastic process to approximate the underlying process $X$. In particular, we can avoid the task of choosing a suitable lag time, and the use of heuristic arguments to justify the choice of a lag time. This feature differentiates our method from methods that perform TPT that use Markov processes on discrete state spaces to approximate the underlying diffusion process of interest.

Third, our method comes with a rigorous error analysis. While there exist rigorous error analyses for Markov state models in the framework of transfer operators on continuous state spaces \cite{sarich2010approximation}, there are no rigorous error analyses for TPT for Markov processes. We prove bounds on the approximation error of each TPT object in terms of how fine the discretisation is. 

The paper proceeds as follows. We describe our discretisation methods in Section \ref{sec:setup} and describe how these lead to a non-Markovian jump process. In Section \ref{sec:convergence_analysis} we define the approximate committor, the approximate probability current, and the approximate streamlines of our method. The approximate committor and the approximate probability current are piecewise constant functions on the continuous state space $S$ of the underlying process $X$, while the approximate streamline is a piecewise linear trajectory in $S$. For each approximate object, we prove a bound on the approximation error with respect to the corresponding TPT object for the underlying process $X$, in terms of the largest diameter of the sets in the discretisation of the continuous state space $S$. We then use these error bounds to prove convergence as the sets in the discretisation shrink to points; this is the above-mentioned `continuum limit'. In Section \ref{sec:numerical_example} we present numerical results, in which we compare the performance of our approach with the finite differences method for TPT for diffusion processes. We conclude in Section \ref{sec:conclusion}.

\section{Setup}
\label{sec:setup}

In what follows, we follow the setup in \cite{weinan2006towards}. Let $X=\{X_t\}_{t\geq 0}$ denote the underlying diffusion process of interest. The state space of $X$ is an open, bounded, $d$-dimensional, simply connected set $S\subset\bb{R}^d$, and the boundary of $S$ is sufficiently regular for us to impose reflecting boundary conditions. We assume that $X$ is ergodic with respect to a probability measure $\mu$, that $\mu$ admits a density with respect to Lebesgue measure on $S$, and that the Lebesgue density of $\mu$ is strictly positive on $S$. We write $A$ and $B$ to denote two open, simply connected subsets of $S$ whose closures are disjoint.

Given an arbitrary set $U\subset S$, let $\inter{U}$, $\clos{U}$, $\partial U$, $U^\complement$ and $\mathbf{1}_U$ denote the interior, closure, boundary, complement with respect to $S$ (i.e. $S\setminus U$), and indicator function of $U$ respectively. The cardinality of an arbitrary set $V$ is written $\card{V}$, and the restriction of a function $f$ to $V$ is written $f\vert_V$. The transpose of a vector  $v\in\bb{R}^{d}$ is written $v^\top$; the $\ell_p$ norm for $1\leq p\leq \infty$ of $v$ is written $\abs{v}_p$. The Lipschitz constant of an arbitrary function $f$ by $\Vrt{f}_{\Lip}$. Given a finite measure $\nu$ on $S$, we denote the $L^p(\nu)$ norm on a space of functions defined on $S$ or a subset of $S$ by $\Vrt{\cdot}_{L^p(\nu)}$. We fix a probability space $(\Omega,\mathcal{F},\mathbb{P})$ that is assumed to be large enough to admit all random variables considered below.

\subsection{Discretisation}
\label{ssec:discretisation}

We review the relevant concepts from Voronoi tessellations and introduce some important terms concerning the discretisations that we will use in the sequel. Our choice of Voronoi tessellations as a discretisation method is motivated by their successful application in not only molecular dynamics \cite{lie_voronoi,venturoli_voronoi} but many other applications; see \cite{voronoi_aurenhammer,voronoi_du} and the references therein. The key property of Voronoi tessellations that we use are that Voronoi tessellations partition a set into polytopes that overlap at most at their boundaries, and that adjacency relations can be computed without computing the full Voronoi tessellation, i.e. without computing the vertices of every polytope in the tessellation.

A Voronoi tessellation of $S$ associated to a finite set of \emph{generators} $\{g_1,\ldots,g_n\}\subset S$ for some $n\in\bb{N}$ is a collection $\{S_1,\ldots,S_n\}$ of nonempty subsets of $S$, where each subset is the Voronoi cell defined by
\begin{equation*}
 S_i:=\{x\in S \ :\ \abs{x-g_i}_2\leq \abs{x-g_j}_2,\ j\neq i\}.
\end{equation*}
Thus, $S_i$ is the closed set consisting of all points in state space that are closer in the Euclidean metric to the generator $g_i$ than to any other generator. It can be shown that each Voronoi cell $S_i\subset \bb{R}^d$ is a $d$-polytope, i.e. a bounded, closed, convex subset such that $\dim{\inter{S_i}}=d$, that is defined as the intersection of finitely many half-spaces or equivalently as the convex hull of finitely many extreme points. The boundedness of each cell $S_i$ follows from the boundedness of the state space $S$. Since every Voronoi cell $S_i$ is a neighbourhood of its generator $g_i$, it has strictly positive Lebesgue measure. The Voronoi cells form a partition of the state space $S$, i.e.
\begin{equation*}
\clos{ S}=\bigcup_{i=1}^{n}S_i,\quad S_i\cap S_j=\partial S_i\cap \partial S_j,\  i\neq j,
\end{equation*}
so that Voronoi cells intersect at most at their boundaries, which have zero Lebesgue measure. We declare two distinct Voronoi cells $S_i$ and $S_j$ to be \emph{adjacent} if they share a common facet, i.e. if $\dim(S_i\cap S_j)=d-1$. Given a Voronoi tessellation $\{S_{i}\}_{i\in I}$, $I=\{1,\ldots,n\}$, its dual Delaunay graph is the graph $G=(I,E)$ with vertex set $I$ and edge set $E$ consisting of all pairs $(i,j)$ such that $S_i$ and $S_j$ are adjacent. 

Recall that, given a nonempty set $A\subset \bb{R}^d$, the Euclidean \emph{diameter} of $A$ is defined by $\diam{A} = \sup\{\abs{x-y}_2\ :\ x,y\in A\}$. This leads to the next definition.
\begin{definition}
\label{definition:width}
 The \emph{width} $\rho$ of a Voronoi tessellation $\{S_i\}_{i\in I}$ is the largest Euclidean diameter of the Voronoi cells, i.e.
 \begin{equation*}
  \rho(\{S_i\}_{i\in I}) := \sup_{i\in I} \diam{S_i}.
 \end{equation*}
\end{definition}
When there is no risk of confusion, we will omit the argument $\{S_i\}_{i\in I}$ of the width and simply write $\rho$. The width provides uniform control over the Euclidean diameter of all the cells in the tessellation. The smaller (respectively larger) the width, the finer (resp. coarser) the tessellation. We shall be interested in obtaining error bounds in the continuum limit, i.e. as $\rho$ decreases to $0$. In the continuum limit, the cells in the Voronoi tessellation shrink to points.

\begin{remark}
 \label{remark:do_not_need_to_compute_vertices_for_normals}
 An important property of Voronoi tessellations is that, if one knows that two Voronoi cells $S_i$ and $S_j$ that are generated respectively by $g_i$ and $g_j$ are adjacent, then the outer unit normal to $S_i$ on the facet $S_i\cap S_j=\partial S_i\cap S_j$ is given by $(g_j-g_i)/\abs{g_i-g_j}_2$. This is because the facet $S_i\cap S_j$ is contained in the hyperplane of points that are equidistant from $g_i$ and $g_j$ and $g_j-g_i$ is normal to this hyperplane.
\end{remark}

\subsection{Definition of approximating jump process}
\label{ssec:def_Y}

Recall that many applications of TPT involve the construction of a Markovian stochastic process on a discrete state space that approximates the underlying diffusion process on the continuous state space. In this section, we construct a continuous-time process on the Delaunay graph $G=(I,E)$ mentioned above. We explain why this process is not Markovian and show that this process preserves an important property of the underlying process $X$ on $S$.

Let $\{S_i\}_{i\in I}$ be a Voronoi tessellation with finite index set $I$, and let $G=(I,E)$ be the Delaunay graph associated to $\{S_i\}_{i\in I}$. We define a continuous time jump process $Y=\{Y_t\}_{t\geq 0}$ on $G$ by `projecting' the trajectory of the underlying process $X$ to $I$ according to
\begin{equation}
\label{eq:def_Y}
Y_t=\begin{cases}
i & X_t\in\inter{S_i},
\\
j & \exists\epsilon>0\text{ s.t. }\forall s\in (t-\epsilon,t),\ X_s\in\inter{S_j}.
\end{cases}
\end{equation}
The second case in \eqref{eq:def_Y} can be interpreted as follows: if $X_t$ lies on the boundary of a Voronoi cell, then we assign to $Y_t$ the state $j\in I$, where $j$ is the index of the set whose interior contained the trajectory of $Y$ in the most recent past, i.e. up to but not including the current time $t$. The memory effect implied by the second case in \eqref{eq:def_Y} of $Y$ implies that $Y$ is not Markovian. 

In order to define the analogues of the sets $A$ and $B$ in the discrete state space $I$, we make the following assumption.
\begin{assumption}
\label{assumption:JK_models_AB}
Given the sets $A,B\subset S$, there exist disjoint subsets $J,K\subset I$ satisfying $\clos{A}=\cup_{j\in J}S_j$ and $\clos{B}=\cup_{k\in K}S_k$.  
\end{assumption}
The assumption implies that $J$ and $K$ are the `metastable subsets' for the jump process $Y$. The assumption effectively amounts to making the definition of $A$ and $B$ dependent on the partition $\{S_i\}_{i\in I}$. Although this may appear to be problematic at first, we can justify it by both theoretical and practical considerations. The theoretical consideration is that we are interested in the continuum limit, in which case the cells $\{S_i\}_{i\in I}$ shrink to points. Thus, even if the assumption is not satisfied for a particular value of the width $\rho$, it will be satisfied in the continuum limit, regardless of the geometry of $A$ or $B$. The practical consideration is that in many applications, there is some flexibility in the definition of the metastable sets $A$ and $B$ in state space. For example, in molecular dynamics, one may choose $A$ and $B$ to each be a set of sufficiently small diameter that contains a local minima of the energy landscape. The precise geometry of each set is less important relative to the preservation of the metastability property. Thus it is reasonable to adapt the definitions of $A$ and $B$ according to the sets in the discretisation, since this facilitates the determination of when a process has entered a metastable set.

Recall the definition of the first hitting time of $X$ with respect to the set $A\cup B$,
\begin{equation}
\label{eq:tau_AB}
\tau_{A\cup B}(X):= \inf \{ t\geq 0\ :\  X_t\in A\cup B \}.
\end{equation}
Define the first hitting time of the jump process $Y$ with respect to the set $J\cup K$ as
\begin{equation}
\label{eq:tau_JK}
 \tau_{J\cup K}(Y) := \inf\{t\geq 0 \ :\ Y_t \in J\cup K\}.
\end{equation}
In practice, it is not possible to detect exactly when $Y$ hits $J$ or $K$, because numerical methods for generating trajectories of the underlying process $X$ will produce only approximations of the trajectories of $X$, e.g. by linear interpolation between points generated by the Euler-Maruyama method. Thus, using a numerically generated trajectory of $X$ in \eqref{eq:def_Y} will generate only an approximation of $Y$. However, for the sake of simplicity, we shall not take into account the error incurred by these approximations, and we shall assume that we can detect when $Y$ hits $J$ or $K$. 

The preceding definitions lead to the following result.
\begin{lemma}
\label{lemma:tau_AB_equals_tau_JK}
Suppose that Assumption \ref{assumption:JK_models_AB} holds, and let $\tau_{A\cup B}(X)$ and $\tau_{J\cup K}(Y)$ be defined as in \eqref{eq:tau_AB} and \eqref{eq:tau_JK} respectively. Then $\tau_{A\cup B}(X)=\tau_{J\cup K}(Y)$.
\end{lemma}
\begin{proof}
Fix an arbitrary sample trajectory $X(\omega)$ of $X$, and let $t:=\tau_{A\cup B}(X(\omega))$. Given that $A$ and $B$ are open, and given the definition \eqref{eq:tau_AB} of $\tau_{A\cup B}(X)$, it follows that there exist $\epsilon_1,\epsilon_2>0$ such that $X_s(\omega)\notin A\cup B$ for $s\in (t-\epsilon_1,t)$, $X_t(\omega)\in\partial A\cup \partial B$, and $X_s(\omega)\in A\cup B$ for $s\in (t,t+\epsilon_2)$. 

Suppose first that $X_t(\omega)\in\partial A$ and $X_s(\omega)\in A$ for $s\in (t,t+\epsilon_2)$. Let $i\in I$ be such that $X_s(\omega)\in\inter{S_i}$ for $s\in (t-\epsilon_1,t)$. By the definition \eqref{eq:def_Y} of $Y$, it follows that $Y_s(\omega)=i$ for $s\in (t-\epsilon_1,t)$ and $Y_t(\omega)=i$. Furthermore, given Assumption \ref{assumption:JK_models_AB}, $X_s(\omega)\in A$ for $s\in (t,t+\epsilon_2)$ implies that $Y_s(\omega)\in J$ for $s\in (t,t+\epsilon_2)$. Since $\inf\{s\ : t<s<t+\epsilon_2\}=t$, it follows that $\tau_{J\cup K}(Y)=t=\tau_{A\cup B}(X)$. If instead $X_t(\omega)\in\partial B$ and $X_s(\omega)\in B$ for $s\in (t,t+\epsilon_2)$, then the same argument holds after replacing $J$ with $K$. This completes the proof.
\end{proof}

Except for $\tau_{J\cup K}(Y)$, our method does not involve computing or estimating any properties of the non-Markovian jump process $Y$, such as its transition probabilities. We introduce the process $Y$ only to help the interpretation of the approximate committor that we define in \S\ref{ssec:approximate_committor}. 

\section{Definition and convergence analysis of approximate TPT objects}
\label{sec:convergence_analysis}

In this section we define the approximate TPT objects that are the outputs of our method. We prove error bounds for each object and use these bounds to show that in the continuum limit, each approximate TPT object we define converges to the corresponding object for the underlying process $X$. For simplicity, in the error bounds we do not take into account statistical errors, i.e. errors due to the use of finite-sample Monte Carlo averages as approximations for expected values. 

\subsection{Approximate committor}
\label{ssec:approximate_committor}

Recall the definition of the committor of $X$,
\begin{equation}
\label{eq:forward_committor}
q:S\to [0,1],\quad    x\mapsto q(x):= \mathbb{P}(X_{\tau_{A\cup B}(X)}\in B\vert X_0 = x).
\end{equation}
Recall that $K$ is the subset of discrete state space $I$ corresponding to the metastable subset $B$ of the continuous state space $S$. We define the approximate committor in an analogous way, using the jump process $Y$ defined in \eqref{eq:def_Y}. For $i\in I$, define $\tilde{q}_i$ by
\begin{equation}
\label{eq:approximate_committor_ith_component}
	\tilde{q}_i := \mathbb{P}\left( Y_{\tau_{J\cup K}(Y)}\in K\ \middle\vert Y_0= i \right).
\end{equation} 
 In order to compute each $\tilde{q}_i$ using sample trajectories of $X$, we use a Monte Carlo approach: we sample finitely many trajectories, where the initial condition of each trajectory is an independent draw from the uniform distribution on $S_i$, and compute the empirical probability that a trajectory starting from $S_i$ enters $B$ before $A$. 
 
We define the approximate committor on the state space $S$ of the underlying process $X$ as a function that is piecewise constant on the interiors of the cells $\{S_i\}_{i\in I}$ of the tessellation:
\begin{equation}
\label{eq:approximate_committor_func}
 \tilde{q}(x):=
 \begin{cases}
  \tilde{q}_i & x\in \inter{S_i}
  \\
  a & x\notin \cup_{i\in I}\inter{S_i},
 \end{cases}
\end{equation}
where we may choose an arbitrary $a\in\bb{R}$, since $(\cup_{i\in I}\inter{S_i})^\complement$ has Lebesgue measure zero and since we will measure the error of $\hat{q}$ with respect to $q$ in the $L^p(\mu)$ norm. Up to a redefinition of the values on $(\cup_{i\in I}\inter{S_i}))^\complement$, the approximate committor $\tilde{q}$ can be interpreted as a committor for $Y$, because the definition \eqref{eq:approximate_committor_func} involves the $\tilde{q}_i$'s from \eqref{eq:approximate_committor_ith_component}.

Recall that the ergodic measure $\mu$ is assumed to have a strictly positive Lebesgue density on $S$. Define the $\mu$-weighted inner product on $L^2(S,\mu;\bb{R})$ by $\langle v,w \rangle_{\mu} = \int_S v(x)w(x)\mu(\mathrm{d}x)$. Given a tessellation $\{S_i\}_{i\in I}$, we use the weighted inner product to define $\hat{q}_i$ for each $i\in I$ by
\begin{equation}
\label{proj_committor}
	\hat{q}_i := \frac{1}{\mu(S_i)}\langle q,\mathbf{1}_{S_i} \rangle_{\mu}.
\end{equation} 
Since each $S_i$ has positive Lebesgue measure, it follows that $\mu(S_i)>0$ for each $i\in I$. We can define the projected committor function $\hat{q}:S\to [0,1]$ analogously to \eqref{eq:approximate_committor_func}:
\begin{equation}
\label{projected_committor_func}
\hat{q}(x):=\begin{cases}
  \hat{q}_i & x\in \inter{S_i}
  \\
  b & x\notin \cup_{i\in I}\inter{S_i},
 \end{cases}
\end{equation}
for a fixed, arbitrary $b\in\bb{R}$. With these preparations, we may state the following proposition. The proof uses some technical lemmas that will not be used in the sequel, so we defer the proof to \S\ref{ssec:appendix_committor}.

\begin{proposition} 
\label{proposition_committors_equal}
Suppose that Assumption \ref{assumption:JK_models_AB} holds. Let $\hat{q}_i$ and $\tilde{q}_i$ be defined as in \eqref{proj_committor} and \eqref{eq:approximate_committor_ith_component} respectively. If $\bb{P}\circ (X_0)^{-1}=\mu$, then $\hat{q}_i = \tilde{q}_i$, for all $i\in I$. 	
\end{proposition}

The following lemma uses the convexity of Voronoi cells and the continuity of the committor function to prove that every Voronoi cell $S_i$ contains a point $x_i\in S_i$ such that $q$ attains the value $\hat{q}_i$ at $x_i$. We shall use this lemma later to prove Theorem \ref{theorem:committor_conv}.
\begin{lemma}
\label{lemma:existence_of_xi}
 Let $\{S_i\}_{i\in I}$ be a Voronoi tessellation of $S$, and let $\hat{q}_i$ be defined as in \eqref{proj_committor}. For every $i\in I$, there exists some $x_i\in S_i$ such that $q(x_i)=\hat{q}_i$.
\end{lemma}
\begin{proof}
If $q$ is constant on $S_i$, then it must equal $\hat{q}_i$, and any $x_i\in S_i$ satisfies the desired property. Therefore, suppose that $q$ is not constant on $S_i$, and partition $S_i$ into the disjoint subsets $S^-_i := \{x\in S_i\ :\ q(x)< \hat{q}_i\}$, $S^+_i := \{x\in S_i\ :\ q(x)> \hat{q}_i\}$ and $S^0_i := \{x\in S_i\ :\ q(x)= \hat{q}_i\}$. Since $q$ is continuous and not constant on $S_i$, it follows that $S^-_i$ and $S^+_i$ are nonempty. Let $x'\in S^-_i$ and $x''\in S^+_i$. It follows from the intermediate value theorem that there exists a $t\in(0,1)$ such that $x_i(t) := (1-t)x' + tx''$ satisfies $q(x_i(t)) = \hat{q}_i$. Since $x',x''\in S_i$ and since any Voronoi cell $S_i$ is convex, it follows that $x_i(t)$ belongs to $S_i$. 
\end{proof}
\begin{remark}
 The conclusion of Lemma \ref{lemma:existence_of_xi} holds for more general partitions $\{S_i\}_{i\in I}$ of state space $S$ for which each $S_i$ is path connected, since in this case we can replace the line segment $x_i(t):=(1-t)x'+tx''$ with a curve in $S_i$ with endpoints $x'$ and $x''$ and follow the same reasoning thereafter.
\end{remark}

We prove an error bound for the error incurred when we approximate the true committor $q$ with the projected committor $\hat{q}$ defined in \eqref{projected_committor_func}.

\begin{theorem}[Error bound for projected committor]
\label{theorem:committor_conv} 
Suppose that the committor $q:S\to [0,1]$ is globally Lipschitz continuous and has bounded second-order derivatives. If $\{S_i\}_{i\in I}$ is a Voronoi tessellation with width $\rho\leq 1$, then for any $p\geq 1$,
\begin{equation*}
\Vert q- \hat{q} \Vert_{L^p(\mu)}\leq C\rho,
\end{equation*}
where $C>0$ depends only on $q$.
\end{theorem}
\begin{proof}
Recall that the committor $q$ is continuously differentiable in $S\setminus (A\cup B)$, since it solves a backward Kolmogorov equation. Given the assumption that $q$ is globally Lipschitz continuous, it follows that $\sup_{x\in S}\abs{\nabla q(x)}\leq \Vrt{q}_{\Lip}$. In addition, given the assumption of bounded second-order derivatives, there exists some $H>0$ that depends only on $q$, such that for all $x,y\in S$ with $\abs{x-y}_2\leq \rho$ and $\rho\leq 1$, 
\begin{equation}\label{taylor_ineq}
    \abs{q(x)-q(y)- \langle \nabla q(x),x-y\rangle} \leq H\abs{x-y}^2_2\leq H\abs{x-y}_2.
\end{equation}
Fix an arbitrary $p\in [1,\infty)$, and fix an arbitrary $i\in I$. By Lemma \ref{lemma:existence_of_xi}, there exists an $x_i\in S_i$ such that $q(x_i)=\hat{q}_i=\hat{q}\vert_{S_i}$. We obtain
	\begin{align*}
	\Vert \left(q - \hat{q}\right)|_{S_i} \Vert^p_{L^p(\mu)} \nonumber &= \int_{S_i} \abs{q(x) - \hat{q}(x)}^p \mu(\mathrm{d}x) = \int_{S_i} \abs{q(x) - q(x_i)}^p \mu(\mathrm{d}x) 
	\\
	\nonumber &= \int_{S_i} \abs{ q(x)-q(x_i) -\langle \nabla q(x_i),x-x_i\rangle + \langle \nabla q(x_i),x-x_i\rangle}^p\mu(\mathrm{d}x)
	\\
	\nonumber &\leq 2^{p-1}\left( \int_{S_i}\left(H\abs{x - x_i}_2 \right)^p\mu(
\mathrm{d}x) + \int_{S_i}\abs{\langle \nabla q(x_i),x-x_i\rangle}^p\mu(\mathrm{d}x) \right)
		\\
		\nonumber &\leq 2^{p-1}\left(\int_{S_i}\left(H\abs{x - x_i}_2 \right)^p\mu(\mathrm{d}x)+\int_{S_i} \left(\Vrt{q}_{\Lip}\abs{x-x_i}_2\right)^p\mu(\mathrm{d}x) \right)
		\\
		\nonumber &\leq 2^{p-1}\rho^p\mu(S_i)\left(H^p  + \Vrt{q}_{\Lip}^p \right).
	\end{align*}
	Above, we used \eqref{taylor_ineq} and $(a+b)^p\leq 2^{p-1}(a^p+b^p)$ in the first inequality, the global Lipschitz continuity of $q$ and the Cauchy-Schwarz inequality in the second inequality, and the assumption that $\{S_i\}_{i\in I}$ has width $\rho$ and the fact that $x,x_i\in S_i$ with $\diam{S_i}\leq \rho$ in the third inequality. Summing over $i\in I$, we obtain
	\begin{align*}
	\Vert q-\hat{q}\Vert_{L^p(\mu)}^p &=\sum_{i\in I}\Vert (q-\hat{q})\vert_{S_i}\Vert_{L^p(\mu)}^p\leq (2\rho(\Vrt{q}_{\Lip}+H))^p
	\end{align*}
    since $\mu(S)=1$, and since $a^p + b^p \leq (a+b)^p$, for $a,b \geq 0$ and $p\geq 1$.
\end{proof}
\begin{corollary}[Error bound for discrete committor]
\label{corollary:approximate_committor_convergence}
Suppose that Assumption \ref{assumption:JK_models_AB} and the assumptions of Theorem \ref{theorem:committor_conv} hold. Then for any Voronoi tessellation $\{S_i\}_{i\in I}$ of $S$ with sufficiently small width $\rho$, the function $\tilde{q}$ defined in \eqref{eq:approximate_committor_func} satisfies
\begin{equation*}
 \Vert q-\tilde{q}\Vert_{L^p(\mu)}\leq C\rho
\end{equation*}
for the same $C>0$ as in the conclusion of Theorem \ref{theorem:committor_conv}.
\end{corollary}
\begin{proof}
 The result follows from Theorem \ref{theorem:committor_conv} and Proposition \ref{proposition_committors_equal}.
\end{proof}

\subsection{Approximate probability current}
\label{ssec:probability_current}

In this section we define the second output of our method, namely the approximate probability current $\tilde{J}_{AB}$. We describe how $\tilde{J}_{AB}$ is obtained from sample reactive trajectories of $X$. 

To define the approximate probability current, we first recall the definition of the probability current from TPT for diffusion processes, from the presentation in \cite[Section 3.1.4]{metzner2008transition}. The probability current is a function $J_{AB} : S\setminus (A\cup B)\rightarrow \mathbb{R}^d$ such that at any point $x\in S\setminus(A\cup B)$, $J_{AB}(x)$ represents the net flux of reactive trajectories from $A$ to $B$ through that point. For any $d$-dimensional $C \subset S\setminus (A \cup B)$ with boundary $\partial C$, $J_{AB}$ is defined implicitly via
\begin{equation}
 \label{eq:probability_current_def}
\begin{split}
\lim_{s\rightarrow 0^+} \frac{1}{s}\lim_{T\rightarrow \infty}\frac{1}{T}\int_{\mathcal{R}\cap \left[ 0,T\right]}\mathbf{1}_{C}\left(X_t\right)\mathbf{1}_{{C}^\complement}\left(X_{t+s}\right) - \mathbf{1}_{{C}^\complement}\left(X_t\right)\mathbf{1}_{C}\left(X_{t+s}\right) \mathrm{d}t \\= \int_{\partial C}J_{AB}(y)\cdot n_{C}(y)\sigma_{C}(\mathrm{d}y)
\end{split}
\end{equation}
where $n_{C}(y)$ denotes the outward facing unit normal to $C$ at some point $y\in \partial C$, $\sigma_{C}$ denotes the surface measure on $\partial C$ and $\mathcal{R}$ indicates the set of reactive times, i.e. the union of time intervals in which the trajectory of the process $X$ is reactive. 

To motivate our definition of the approximate probability current, we first observe that if we set the region $C$ in the formula above to be a Voronoi cell or more generally, a polytope $S_i$, then the unit normal $n_C$ in the integral on the right-hand side of \eqref{eq:probability_current_def} takes only finitely many values that are given by the outer unit normals to the facets of $S_i$. In particular, we have
\begin{equation*}
\int_{\partial S_i}J_{AB}(y)\cdot n_{S_i}(y)\sigma_{S_i}(\mathrm{d}y) = \sum_{k\in \mathcal{N}(S_i)}\int_{\partial{S_i}\cap \partial{S_k}}J_{AB}(y)\cdot n_{ik}(y)\sigma_{\partial S_i \cap \partial S_k}(\mathrm{d}y), 
\end{equation*}
where $\mathcal{N}(S_i)$ denotes the set of indices of cells adjacent to $S_i$, i.e.
\begin{equation}
 \label{eq:neighbour_set}
 \mathcal{N}(S_i):=\{j\in I\ :\ \dim{S_i\cap S_j}=d-1\},
\end{equation}
and $n_{ik}$ is the outer unit normal vector to $S_i$ in the relative interior of the facet $\partial S_i\cap\partial S_k$. Thus we arrive at a decomposition over facets of the right-hand side of \eqref{eq:probability_current_def}. 

We can obtain a similar decomposition over facets of the left-hand side of \eqref{eq:probability_current_def}. Observe that if in a time interval of vanishing length a transition out of $C=S_i$ takes place, then almost surely with respect to normalised surface measure on $\partial S_i$, the transition is into an adjacent cell $S_k$. This is due to the almost sure continuity of sample paths of diffusion processes. Therefore, in the $s\to 0^+$ limit, $\mathbf{1}_{C^\complement}(X_{t+s})=\sum_{k\in\mathcal{N}(S_i)}\mathbf{1}_{S_k}(X_{t+s})$. Thus, if $C=S_i$ for arbitrary $i\in I$, then we can write \eqref{eq:probability_current_def} as a system of $\card{\mathcal{N}(S_i)}$ equations: for $k\in\mathcal{N}(S_i)$,
\begin{equation}\label{eq:alpha_ik_def}
 \begin{split}
\alpha_{ik}:=&\lim_{s\rightarrow 0^+} \frac{1}{s}\lim_{T\rightarrow \infty}\frac{1}{T}\int_{\mathcal{R}\cap \left[ 0,T\right]}\mathbf{1}_{S_i}\left(X_t\right)\mathbf{1}_{S_k}\left(X_{t+s}\right) - \mathbf{1}_{S_k}\left(X_t\right)\mathbf{1}_{S_i}\left(X_{t+s}\right) \mathrm{d}t 
\\
=& \int_{\partial{S_i}\cap \partial{S_k}}J_{AB}(y)\cdot n_{ik}\sigma_{\partial S_i \cap \partial S_k}(\mathrm{d}y).
\end{split}
\end{equation}
Following the interpretation of the probability current $J_{AB}(y)$ at $y$ as the net flux of reactive trajectories from $A$ to $B$ through $y$, we may interpret $\alpha_{ik}$ as the average net flux of reactive trajectories across the facet $S_i\cap S_k$.

We now describe how to obtain the approximate probability current. Under reasonable conditions, the probability current is continuous on its domain \cite[Section 3.1.4, Eq. (3.15)]{metzner2008transition}. If the probability current is continuous, then for sufficiently small $S_i$, we may replace the probability current $J_{AB}$ in the right-hand side of \eqref{eq:alpha_ik_def} with a constant vector $\tilde{J}_{AB,i}\in\bb{R}^d$. Since the choice of $\tilde{J}_{AB,i}$ should be consistent with \eqref{eq:alpha_ik_def}, the resulting expression yields the system of equations for the unknown $\tilde{J}_{AB,i}$:
\begin{equation}
\label{eq:hat_alpha_ik}
\hat{\alpha}_{ik}:=\frac{\alpha_{ik}}{\sigma(\partial{S_i} \cap \partial{S_k})} = \tilde{J}_{AB,i}\cdot n_{ik},\quad k\in\mathcal{N}(S_i).
\end{equation}
Note that in order to compute the volume $\sigma(\partial{S_i} \cap \partial{S_k})$ of the facet $\partial{S_i} \cap \partial{S_k}$, we need the vertices of the facet. We will return to this observation in \S\ref{sec:conclusion}. 

Let $\hat{\alpha}_i\in \mathbb{R}^{\card{\mathcal{N}(S_i)}}$ denote the vector with entries given by $\hat{\alpha}_{ik}$, and let $N_i\in\mathbb{R}^{\card{\mathcal{N}(S_i)}\times d}$ denote the matrix whose $k$-th row is given by $n_{ik}^\top$. Then we may rewrite the preceding system of equations as a matrix-vector product
\begin{equation}\label{eq:def_JAB_i_tilde}
	\hat{\alpha}_i =N_i\tilde{J}_{AB,i}\in \bb{R}^{\card{\mathcal{N}(S_i)}}.
\end{equation}
For \eqref{eq:def_JAB_i_tilde} to admit a solution $\tilde{J}_{AB,i}$, we need that $\hat{\alpha}_i$ belongs to the column space of $N_i$. Since $\card{\mathcal{N}(S_i)}\geq d+1$ (see Lemma \ref{lemma:d_polytopes_have_dplus1_facets} below) this need not hold in general. Therefore we consider the system of normal equations
\begin{equation}\label{eq:normal_equation}
M_i\tilde{J}_{AB,i}=\beta_i,\quad M_i:=N_i^\top N_i\in\bb{R}^{d\times d},\quad \beta_i:=N_i^\top \hat{\alpha}_i\in \bb{R}^{d}
\end{equation}
and use the fact that for any $d$-dimensional polytope in $\bb{R}^d$, the matrix $N_i$ of outer unit normals to $S_i$ has full rank (see Corollary \ref{corollary:rank_of_any_cell_is_d}). Thus $M_i$ is invertible, and we may solve \eqref{eq:normal_equation} using standard methods from numerical linear algebra. By solving \eqref{eq:normal_equation} for every $i\in I$, we obtain the approximate probability current $\tilde{J}_{AB}$ as a piecewise constant function on $S\setminus (A\cup B)$, by setting
\begin{equation}
\label{eq:approximate_probability_current_def}
   \tilde{J}_{AB}(x)=\tilde{J}_{AB,i(x)},\quad i(x):=\min\left\{\argmax\left\{\abs{\tilde{J}_{AB,i}}_2\ \middle\vert\ x\in S_i \right\}\right\}.
 \end{equation}
That is, if $x\in \inter{S_i}$ then $\tilde{J}_{AB}(x)=\tilde{J}_{AB,i}$, and if $x\in\partial S_i$ for certain $i\in I$, then we set $\tilde{J}_{AB}(x)$ to be the $\tilde{J}_{AB,i}$ with largest $\ell_2$ norm and choose the vector with smallest index if there are two or more vectors with the same $\ell_2$ norm.
 \begin{remark}
  \label{remark:approximate_probability_current_def_arbitrariness}
   For the error bound that we prove in this section, the values of the approximate probability current on the boundaries of cells are not important. This explains the arbitrary nature of the choice we made above. However, the property that the value is chosen from the set of approximate probability current vectors of cells whose boundaries contain $y$ ensures that the boundary values are reasonable.
 \end{remark}

We outline how to compute each entry of $\hat{\alpha}_{ik}$. Suppose we use a numerical method for generating trajectories of $X$ with time step $0<\Delta t\ll 1$, and suppose that we have generated one long trajectory of $X$ of duration $T:=M\Delta t$ for some $M\in\bb{N}$, $M\gg 1$. For a given $i\in I$ and $k\in\mathcal{N}(S_i)$, we compute $\hat{\alpha}_{ik}$ by taking the difference between the total number of observed one-step transitions from $S_i$ to $S_k$ and the total number of observed one-step transitions from $S_k$ to $S_i$, where we count only those one-step transitions that occur within reactive trajectory segments of the long trajectory. Since we cannot observe state transitions that occur over intervals of length shorter than $\Delta t$, we set $s=\Delta t$ and divide this difference by $T$ and $s$ as in \eqref{eq:alpha_ik_def}. Note that the $\Delta t$ term in $T=M\Delta t$ cancels the $\Delta t$ component arising from the $\mathrm{d}t$ term in the integral. The procedure just described yields $\alpha_{ik}$. We can compute the surface measure $\sigma(\partial S_i\cap\partial S_k)$ using methods described in \cite{bueler_enge_fukuda2000}, and then use \eqref{eq:hat_alpha_ik}.

Modulo statistical error due to insufficient sampling and the error due to the approximate volume computations needed for $\sigma(\partial S_i\cap \partial S_k)$, the computation method described above incurs errors due to the approximation of the $s\to 0^+$ limit and of the $T\to\infty$ limit by a fixed time step and duration respectively. These constraints are inevitable, given a finite computational budget. In particular, the error due to the approximation of the $T\to\infty$ limit is the finite sampling error that is common to all Monte Carlo methods.

For any real matrix $G$, we denote the smallest and largest nonzero singular value of $G$ respectively by $\sigma_{\min}(G)$ and $\sigma_{\max}(G)$. We define the smallest and largest singular value of a Voronoi cell $S_i$ as the smallest and largest singular value of the matrix $N_i$ of outer unit normals respectively:
\begin{equation}
\label{eq:extreme_singular_values_of_polytopes}
\sigma_{\min}(S_i):=\sigma_{\min}(N_i) = \sqrt{\sigma_{\min}(M_i)},\quad\sigma_{\max}(S_i):=\sigma_{\max}(N_i) = \sqrt{\sigma_{\max}(M_i)}.
\end{equation}
With this notation in mind, we can state and prove the following result. 
\begin{theorem}
 \label{theorem:J_AB_error_bound}
 Let $\{S_i\}_{i\in I}$ be a Voronoi tessellation of $S$ with width $\rho$ that satisfies Assumption \ref{assumption:JK_models_AB}. If the probability current $J_{AB}$ is globally Lipschitz, then 
	 \begin{equation}
	  \label{eq:error_bound_approx_J_AB_fixed_tessellation}
	  \Vrt{\tilde{J}_{AB}-J_{AB}}_{L^2(S\setminus (A\cup B),\mu;\bb{R}^d)}\leq C\Vrt{J_{AB}}_{\Lip}\rho,
	 \end{equation}
	 for $C=C(\{S_i\}_{i\in I}):=\sup_{i\in I}\sigma_{\max}(S_i)\sigma_{\min}^{-2}(S_i)\card{\mathcal{N}(S_i)}^{1/2}$.
\end{theorem}
\begin{proof}
We first prove a pointwise error bound. For an arbitrary $x\in S\setminus (A\cup B)$,
	 \begin{equation}\label{eq:error_bound_approx_J_AB_pointwise}
	 \abs{\tilde{J}_{AB,i(x)} - J_{AB}(x)}_2 \leq \frac{\sigma_{\max}(S_{i(x)})}{\sigma_{\min}^2(S_{i(x)})}\sqrt{\card{\mathcal{N}(S_{i(x)})}} \Vrt{J_{AB}}_{\Lip}\rho,
	 \end{equation}
for the function $i:S\to I$ defined in \eqref{eq:approximate_probability_current_def}. Recall that the matrix $M_i$ from \eqref{eq:normal_equation} is invertible. Using \eqref{eq:normal_equation} and \eqref{eq:extreme_singular_values_of_polytopes}, we obtain
\begin{align*}
\abs{\tilde{J}_{AB,i(x)} - J_{AB}(x)}_2 &= \abs{M_{i(x)}^{-1}N_{i(x)}^\top N_{i(x)}\left(\tilde{J}_{AB,{i(x)}} - J_{AB}(x) \right)}_2
\nonumber
\\ 
&\leq \sigma_{\max}\left({M_{i(x)}}^{-1} \right)\sigma_{\max}(N_{i(x)}) \abs{ N_{i(x)}\left( \tilde{J}_{AB,i(x)} - J_{AB}(x) \right)}_2 
\\
&=\frac{\sigma_{\max}(S_{i(x)})}{\sigma_{\min}^2(S_{i(x)})}\abs{ N_{i(x)}\left( \tilde{J}_{AB,i(x)} - J_{AB}(x) \right)}_2.
\end{align*}
From \eqref{eq:def_JAB_i_tilde} we obtain
\begin{equation*}
\abs{N_{i(x)}(\tilde{J}_{AB,i(x)} - J_{AB}(x))}_2^2=\sum_{j\in\mathcal{N}(S_{i(x)})}\abs{\hat{\alpha}_{i(x)j} - n_{i(x)j}\cdot J_{AB}(x)}^2.
\end{equation*}
For any $i\in I$, $x\in S_i$, and $j\in\mathcal{N}(S_i)$,
\begin{align*}
 \abs{\hat{\alpha}_{ij}-n_{ij}\cdot J_{AB}(x)}&=\abs{\frac{1}{\sigma(\partial{S_i}\cap \partial{S_j})}\int_{\partial{S_i}\cap \partial{S_j}} n_{ij}\cdot \left(J_{AB}(y) - J_{AB}(x)\right)\sigma_{\partial S_i \cap \partial S_j}(\mathrm{d}y)}
 \\
 &\leq \frac{1}{\sigma(\partial{S_i}\cap \partial{S_j})}\int_{\partial{S_i}\cap \partial{S_j}} \abs{n_{ij}}_2\abs{J_{AB}(y) - J_{AB}(x)}_2\sigma_{\partial S_i \cap \partial S_j}(\mathrm{d}y)
 \\
 &\leq \frac{1}{\sigma(\partial{S_i}\cap \partial{S_j})}\int_{\partial{S_i}\cap \partial{S_j}} \Vrt{J_{AB}}_{\Lip}\abs{y - x}_2\sigma_{\partial S_i \cap \partial S_j}(\mathrm{d}y)
 \\
 &\leq \Vrt{J_{AB}}_{\Lip}\rho.
\end{align*}
Above, we used \eqref{eq:alpha_ik_def} and \eqref{eq:def_JAB_i_tilde} for the equation, the Cauchy-Schwarz inequality for the first inequality, the Lipschitz continuity assumption of $J_{AB}$ and the fact that for every $i\in I$ and $j\in\mathcal{N}(S_i)$ each $n_{ij}$ has unit norm in the second inequality, and the fact that $\{S_i\}_{i\in I}$ has width $\rho$ implies that for $x\in S_i$ and $y\in \partial S_i$, $\vert y-x\vert_2\leq \rho$ in the last inequality. Combining the preceding inequalities yields \eqref{eq:error_bound_approx_J_AB_pointwise}.

By Assumption \ref{assumption:JK_models_AB}, the union of the $S_i$ with $i\in I\setminus(J\cup K)$ covers $S\setminus (A\cup B)$, so
\begin{equation*}
 \Vrt{\tilde{J}_{AB}-J_{AB}}^2_{L^2(S\setminus (A\cup B),\mu;\bb{R}^d)}=\sum_{i\in I\setminus (J\cup K)}\int_{\inter{S_i}}\abs{\tilde{J}_{AB,i} - J_{AB}(x)}_2^2\mu(\mathrm{d}x).
\end{equation*}
Above, we use the fact that the boundaries of the cells have Lebesgue measure zero  and hence $\mu$-measure zero, the fact the $\{\inter{S_i}\}_{i\in I}$ are disjoint, and the fact that $x\in \inter{S_i}$ implies $i(x)=i$. Thus \eqref{eq:error_bound_approx_J_AB_fixed_tessellation} follows from \eqref{eq:error_bound_approx_J_AB_pointwise}, since $\mu(S)= 1$. 
\end{proof}

To establish a convergence result in the continuum limit, we need to specify a sequence $(\rho_k)_{k\in\bb{N}}$ of widths that decreases to zero, and to impose uniform control over the corresponding sequence of tessellation-dependent coefficients $C$ from \eqref{eq:error_bound_approx_J_AB_fixed_tessellation}.
\begin{corollary}
 \label{corollary:approximate_probability_current_convergence}
 Let $(\rho_k)_{k\in\mathbb{N}}\subset (0,\infty)$ satisfy $\rho_k\to 0$. For each $k\in\bb{N}$, let $\{S^k_i\}_{i\in I(k)}$ be a partition of $S$ with index set $I(k)$, width $\rho_k$, and corresponding approximate probability current $\tilde{J}^k_{AB}$, that satisfies Assumption \ref{assumption:JK_models_AB}. If there exists $0<C<\infty$ that does not depend on $k\in\mathbb{N}$ such that
 \begin{equation}
 \label{eq:uniform_bound_tessellation_constants}
  \sup_{k\in \bb{N}}\sup_{i\in I(k)}\sigma_{\max}(S^k_i)\sigma_{\min}^{-2}(S^k_i)\card{\mathcal{N}(S_i)}^{1/2}\leq C
 \end{equation}
 and if the probability current $J_{AB}$ is globally Lipschitz, then 
 \begin{equation*}
  \Vrt{\tilde{J}^k_{AB}-J_{AB}}_{L^2(S\setminus (A\cup B),\mu;\bb{R}^d)}\leq C\Vrt{J_{AB}}_\Lip\rho_k.
 \end{equation*}
\end{corollary}
\begin{proof}
 For each $k\in\bb{N}$, Theorem \ref{theorem:J_AB_error_bound} yields the upper bound on the error 
 \begin{align}
  \Vert \tilde{J}^k_{AB}-J_{AB}\Vert_{L^2(S\setminus (A\cup B),\mu;\bb{R}^d)}  \leq C_k\Vrt{J_{AB}}_\Lip\rho_k,
 \nonumber
 \\
 \label{eq:C_k}
 C_k:=\sup_{i\in I(k)}\sigma_{\max}(S^k_i)\sigma_{\min}^{-2}(S^k_i)\card{\mathcal{N}(S_i)}^{1/2}.
\end{align}
Since $\sup_{k\in\bb{N}}C_k\leq C$ by \eqref{eq:uniform_bound_tessellation_constants}, the conclusion follows.
\end{proof}
A sequence of Voronoi tessellations that satisfies \eqref{eq:uniform_bound_tessellation_constants} is the sequence of tessellations where each sets in the tessellation is a translation and scaling of the hypercube in $\bb{R}^d$. In this case, for every $i\in I(k)$ we have that $\card{\mathcal{N}(S_i)}=2d$, and the matrix of outer unit normals $N_i$ for each $S^k_i$ satisfies $N_i^\top N_i=2I_d$, where $I_d$ is the $d\times d$ identity matrix. However, since the number of sets in the resulting hypercubic partition of $S$ increases exponentially with the dimension $d$, such a tessellation would not be suitable when $d\gg 1$. The question of designing algorithms that produce discretisations satisfying \eqref{eq:uniform_bound_tessellation_constants} is interesting, but beyond the scope of this paper.

In this section, the only objects that require computing the vertices of every Voronoi cell are the volumes $\sigma(\partial S_i\cap\partial S_k)$ of the facets. These volumes are needed in \eqref{eq:hat_alpha_ik}. By Remark \ref{remark:do_not_need_to_compute_vertices_for_normals}, the full Voronoi tessellation is not necessary to compute the matrices $\{N_i\}_{i\in I}$ of outer unit normals.

\subsection{Approximate streamlines}
\label{ssec:approximate_streamlines}

In \S\ref{sec:intro}, we described the streamlines as integral curves of the probability current $J_{AB}$. That is, a streamline starting at $s_0\in\partial A$ and ending at $\partial B$ is the solution $\{s(t)\}_{t\in [0,T(s_0)]}$ of a differential equation with $J_{AB}$ as the driving vector field. We can write this streamline as 
\begin{subequations}
\label{eq:streamline}
\begin{align}
\label{eq:true_streamline_alternate_definition}
s(t)&=s_0+\int_0^{t}J_{AB}(s(r))\mathrm{d}r, \quad 0\leq t\leq T(s_0),
\\
 T(s_0)&:=\inf\{t>0\ :\ s(t)\in B\}.
 \label{eq:streamline_time_T} 
\end{align}
\end{subequations}
Note that for every $s_0\in\partial A$, $T(s_0)$ is finite. If this were not the case, then one can obtain a contradiction by using \eqref{eq:probability_current_def} to show that there exist reactive trajectories that do not reach $B$.

For a given tessellation $\{S_i\}_{i\in I}$ with index subsets $K$ and $J$ from Assumption \ref{assumption:JK_models_AB} with corresponding approximate probability current $\tilde{J}_{AB}$ defined in \eqref{eq:approximate_probability_current_def}, we define the approximate streamline with initial condition $\tilde{s}_0\in\partial (\cup_{i\in J}S_i)$ by
\begin{subequations}
 \label{eq:approximate_streamline}
 \begin{align}
 \label{eq:approximate_streamline_alternate_definition}
 \tilde{s}(t)&=\tilde{s}_0+\int_0^{t}\tilde{J}_{AB}(\tilde{s}(r))\mathrm{d}r,\quad 0\leq t\leq \tilde{T}(\tilde{s}_0),
 \\
 \tilde{T}(\tilde{s}_0)&:=\inf\left\{t>0\ :\ \tilde{s}(t)\in \partial\left( \cup_{i\in K}S_i\right)\right\}.
 \label{eq:approximate_streamline_time_tilde_T}
 \end{align}
\end{subequations}
Note that $\tilde{T}(\tilde{s}_0)$ need not be finite for every $\tilde{s}_0\in \partial(\cup_{i\in J}S_i)$: we cannot apply the same reasoning as we did for showing that $T(s_0)$ is finite, because $\tilde{J}_{AB}$ is only an approximation of $J_{AB}$.

\begin{theorem}
\label{theorem:error_bound_for_discrete_streamlines}
 Let $(\rho_k)_{k\in\mathbb{N}}\subset (0,\infty)$ satisfy $\rho_k\to 0$. For each $k\in\bb{N}$, let $\{S^k_i\}_{i\in I(k)}$ be a partition of $S$ with width $\rho_k$ that satisfies Assumption \ref{assumption:JK_models_AB}. Fix $s_0\in\partial A$. For each $k\in\bb{N}$, let $\tilde{s}^k$ denote the approximate streamline corresponding to $\{S^k_i\}_{i\in I(k)}$ with $\tilde{s}^k(0)=s_0$. If $J_{AB}$ is globally Lipschitz continuous and there exists $0<C<\infty$ that does not depend on $k\in\mathbb{N}$ such that \eqref{eq:uniform_bound_tessellation_constants} holds, then there exists $C'>0$ that depends only on $C$, $\Vert J_{AB}\Vert_{\Lip}$ and $s_0$, such that for all $k\in\bb{N}$,
 \begin{equation*}
     \Vrt{\tilde{s}^k - s}_{L^\infty([0,T(s_0)],\mathrm{d}t;\mathbb{R}^d)}\leq C' \rho_k.
     \end{equation*}
\end{theorem} 
Note that the error bound for the approximate probability current in Theorem \ref{theorem:J_AB_error_bound} is computed using the $L^2$ norm, while the error bound for the approximate streamline in Theorem \ref{theorem:error_bound_for_discrete_streamlines} is computed using the $L^\infty$ norm. This is because we are interested in the maximal deviation of the approximate streamlines $\tilde{s}^k$ from the true streamline, given that they have a common initial condition.
\begin{proof}
Fix an arbitrary $k\in\mathbb{N}$, and let $\tilde{J}^k_{AB}$ be the approximate probability current corresponding to $\{S^k_i\}_{i\in I(k)}$ that generates the approximate streamline $\tilde{s}^k$ via \eqref{eq:approximate_streamline_alternate_definition} and the initial condition $\tilde{s}^k(0)=s_0$. It follows from \eqref{eq:true_streamline_alternate_definition} that
\begin{equation*}
 \abs{s(t)-\tilde{s}^k(t)}_\infty\leq \int_0^t\abs{J_{AB}(s(r))-\tilde{J}^k_{AB}(\tilde{s}^k(r))}_\infty\mathrm{d}r,\quad 0\leq t\leq T(s_0).
\end{equation*}
By the triangle inequality and Lipschitz continuity of $J_{AB}$, we have
\begin{align*}
 &\abs{J_{AB}(s(r))-\tilde{J}^k_{AB}(\tilde{s}^k(r))}_\infty 
 \\
 &\leq \abs{J_{AB}(s(r))-J_{AB}(\tilde{s}^k(r))}_\infty+\abs{J_{AB}(\tilde{s}^k(r))-\tilde{J}^k_{AB}(\tilde{s}^k(r))}_\infty
 \\
 &\leq \Vrt{J_{AB}}_{\Lip}\abs{s(r)-\tilde{s}^k(r)}_\infty+\abs{J_{AB}(\tilde{s}^k(r))-\tilde{J}^k_{AB}(\tilde{s}^k(r))}_\infty,
\end{align*}
for all $0\leq r\leq T(s_0)$. We bound the second term on the right-hand side of the last inequality: for all $0\leq r\leq T(s_0)$,
\begin{align*}
 \abs{J_{AB}(\tilde{s}^k(r))-\tilde{J}^k_{AB}(\tilde{s}^k(r))}_\infty= & \abs{J_{AB}(\tilde{s}^k(r))-\tilde{J}^k_{AB,i(\tilde{s}^k(r))}}_\infty
 \\
 \leq&  \abs{J_{AB}(\tilde{s}^k(r))-\tilde{J}^k_{AB,i(\tilde{s}^k(r))}}_2\leq C_k\Vrt{J_{AB}}_{\Lip}\rho_k,
\end{align*}
where the equation follows from \eqref{eq:approximate_probability_current_def}, the first inequality follows from the inequality $\abs{x}_\infty\leq \abs{x}_2$ that is valid for any finite-dimensional vector $x$, and the second inequality follows from \eqref{eq:error_bound_approx_J_AB_pointwise} with $C_k$ defined in \eqref{eq:C_k}. Thus for $0\leq t\leq T(s_0)$,
\begin{align*}
 \abs{s(t)-\tilde{s}^k(t)}_\infty &\leq \Vrt{J_{AB}}_{\Lip}\int_0^t \left(\abs{s(r)-\tilde{s}^k(r)}_\infty+C\rho_k\right)\mathrm{d}r
 \\
 &\leq  \Vrt{J_{AB}}_{\Lip}\left(\int_0^t \abs{s(r)-\tilde{s}^k(r)}_\infty\mathrm{d}r+C_kt\rho_k \right),
\end{align*}
and by the Gronwall-Bellman inequality, we obtain the error bound
\begin{equation*}
%\label{eq:error_bound_for_discrete_streamlines}
 \abs{s(t)-\tilde{s}^k(t)}_\infty \leq C_k\Vrt{J_{AB}}_{\Lip}T(s_0) e^{T(s_0)\Vrt{J_{AB}}_{\Lip}} \rho_k,\quad 0\leq t\leq T(s_0).
\end{equation*}
Since there exists $0<C<\infty$ such that \eqref{eq:uniform_bound_tessellation_constants} holds, the proof is complete.
\end{proof}

Note that the definition of the approximate streamlines does not permit an interpretation in terms of $Y$. This definition allows for two approximate streamlines that lie in the same cell $S_i$ at possibly different times to subsequently enter different cells, e.g. if the streamlines exit $S_i$ via different facets. In terms of the Delaunay graph $G=(I,E)$ of the Voronoi tessellation, this means that it is possible for $\tilde{s}_1(t_1)=\tilde{s}_2(t_2)=i$ for possibly different $t_1,t_2$ and $\tilde{s}_1(t'_1)\neq \tilde{s}_2(t'_2)$ for $t'_1>t_1$, $t'_2>t_2$. This non-uniqueness property prevents an interpretation of the discrete streamlines in $G$ for $Y$ as being analogues of the streamlines in $S$ for $X$, because standard results from the theory of ordinary differential equations guarantee that the streamlines of $X$ are uniquely determined by the probability current and initial condition. Since the approximate streamlines are integral curves of the approximate probability current, this implies in addition that the approximate probability current does not permit an interpretation as a probability current for $Y$ on $G$.

\section{Numerical results}
\label{sec:numerical_example}
In this section we implement our method on a well-known example. We focus on the approximate committor and the approximate probability current only, and study the dependence of the errors of these objects as a function of the tessellation width parameter $\rho$. 

Let the underlying diffusion process $X$ be the solution to the Smoluchowski SDE
\begin{equation}
\label{eq:Smoluchowski}
	\mathrm{d}{X}_t =- \Gamma^{-1}\nabla V(X_t)\mathrm{d}t + \sqrt{2\beta^{-1}}\Gamma^{-\frac{1}{2}}\mathrm{d}W_t, 
\end{equation}
where $V: \bb{R}^d \rightarrow \bb{R}$ denotes the potential function, $\beta$ is an inverse temperature parameter, i.e. $\beta  =(k_B \mathcal{T})^{-1}$ with Boltzmann constant denoted by $k_B$ and temperature $\mathcal{T}$, and $\Gamma \in \bb{R}^{d \times d}$ is a diagonal matrix with friction coefficients on the diagonal.

\subsection{Setup, discretisation, trajectory sampling}

In our example, $d = 2$, and $V$ is the triple well potential from \cite{park2003reaction} given by 
\begin{align}\label{eq:potential_func}
    V(x_1,x_2) &= 3e ^{-x_1^2 - \left(x_2 - \frac{1}{3} \right)^2} - 3e^{-x_1^2 - \left(x_2 - \frac{5}{3} \right)^2} - 5 e^{-(x_1-1)^2 - x_2^2} - 5e^{-(x_1+1)^2 - x_2^2} \nonumber \\
    & + \frac{1}{5}x_1^4 + \frac{1}{5}\left(x_2 - \frac{1}{3}\right)^4,
\end{align}
on the state space $S = \{x=(x_1,x_2)\in \bb{R}^2 \hskip1ex \vert \hskip1ex -2\leq x_1 \leq 2, \hskip1ex -1.5\leq x_2 \leq 2.5\}$, as shown in Figure \ref{figure:energy_landscape}. We define the metastable sets to be the sublevel sets of $V$ that contain the minima at $(\pm 1,0)$, i.e. $A:= \{(x_1,x_2)\in S \hskip1ex \vert \hskip1ex V(x_1,x_2) \leq -3,\hskip1ex x_1\leq 0\}$ and $B:= \{(x_1,x_2)\in S\hskip1ex \vert \hskip1ex V(x_1,x_2)\leq -3,\hskip1ex x_1 \geq 0\}$. 

For the discretisation of $S$ we used a mesh of side length $h$, where the values of $h$ were chosen from $\{0.5, 0.4, 0.25, 0.1,0.05\}$. Letting $[N]:=\{1,2,\ldots, N\}$, it follows that the corresponding discrete state spaces $I$ were $\{[64],[100],[256],[1600],[6400]\}$, that any cell in any tessellation has at most four adjacent cells, and that the resulting Voronoi tessellations each have width $h\sqrt{2}$. With regards to Assumption \ref{assumption:JK_models_AB}, for each tessellation, we defined $i\in J$ (respectively $i\in K$) if at least one vertex of the square $S_i$ belongs to $A$ (resp. $B$). This results in a representation of the sets $A$ and $B$ by the sets $\cup_{j\in J}S_j$ and $\cup_{k\in K}S_k$ respectively. An example of such a tessellation and the sets corresponding to $J$ and $K$ is shown in Figure \ref{figure:energy_landscape_top_view_mesh}.

\begin{figure*}[t!]
    \centering
    \begin{subfigure}[t]{0.475\textwidth}
        \includegraphics[width=1\linewidth]{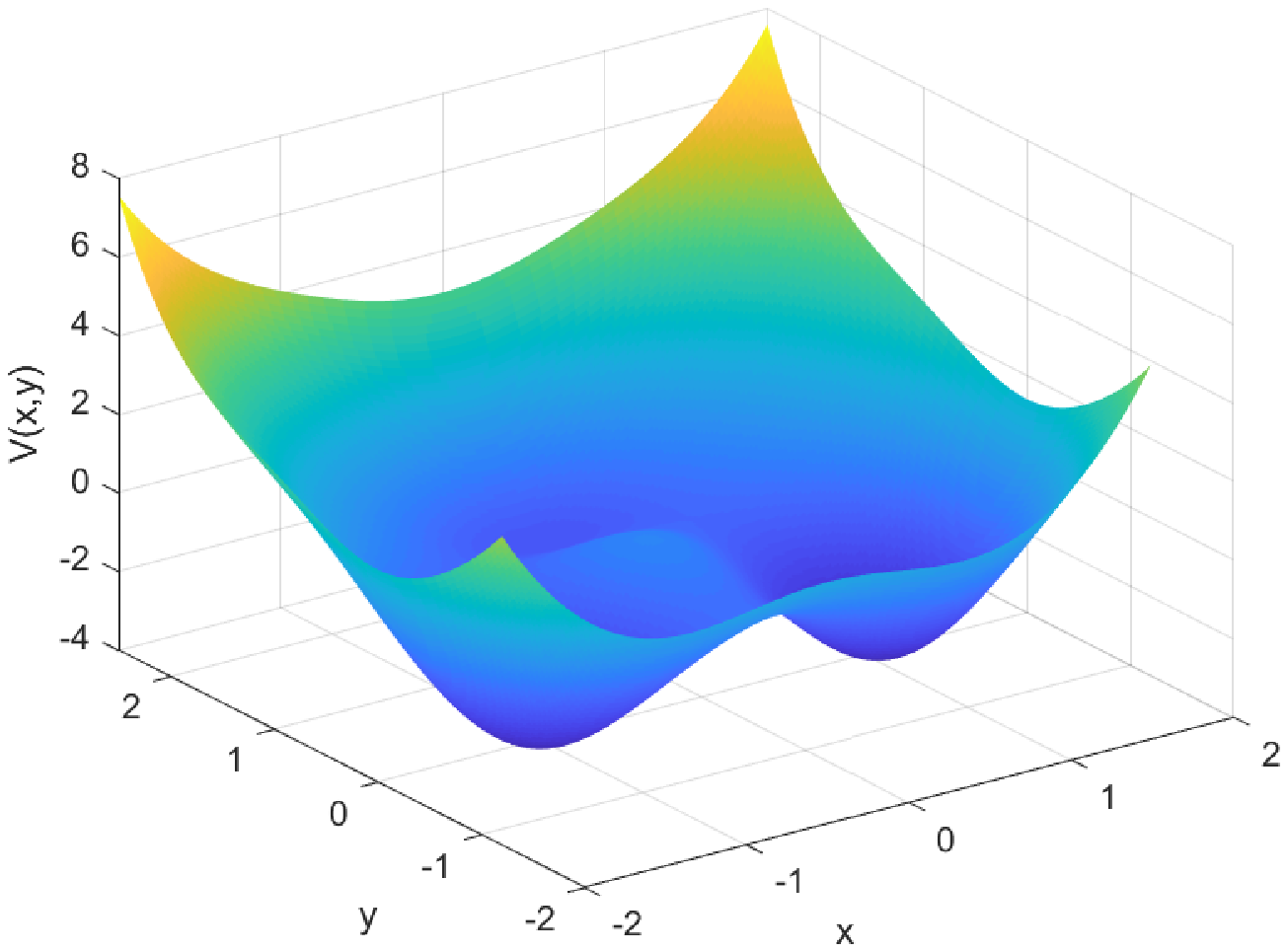}
        \caption{Graph of the potential $V$ given in \eqref{eq:potential_func}.}
    \label{figure:energy_landscape}
    \end{subfigure}%
    \hskip2ex
    \begin{subfigure}[t]{0.475\textwidth}
        \centering
        \includegraphics[width=1\linewidth]{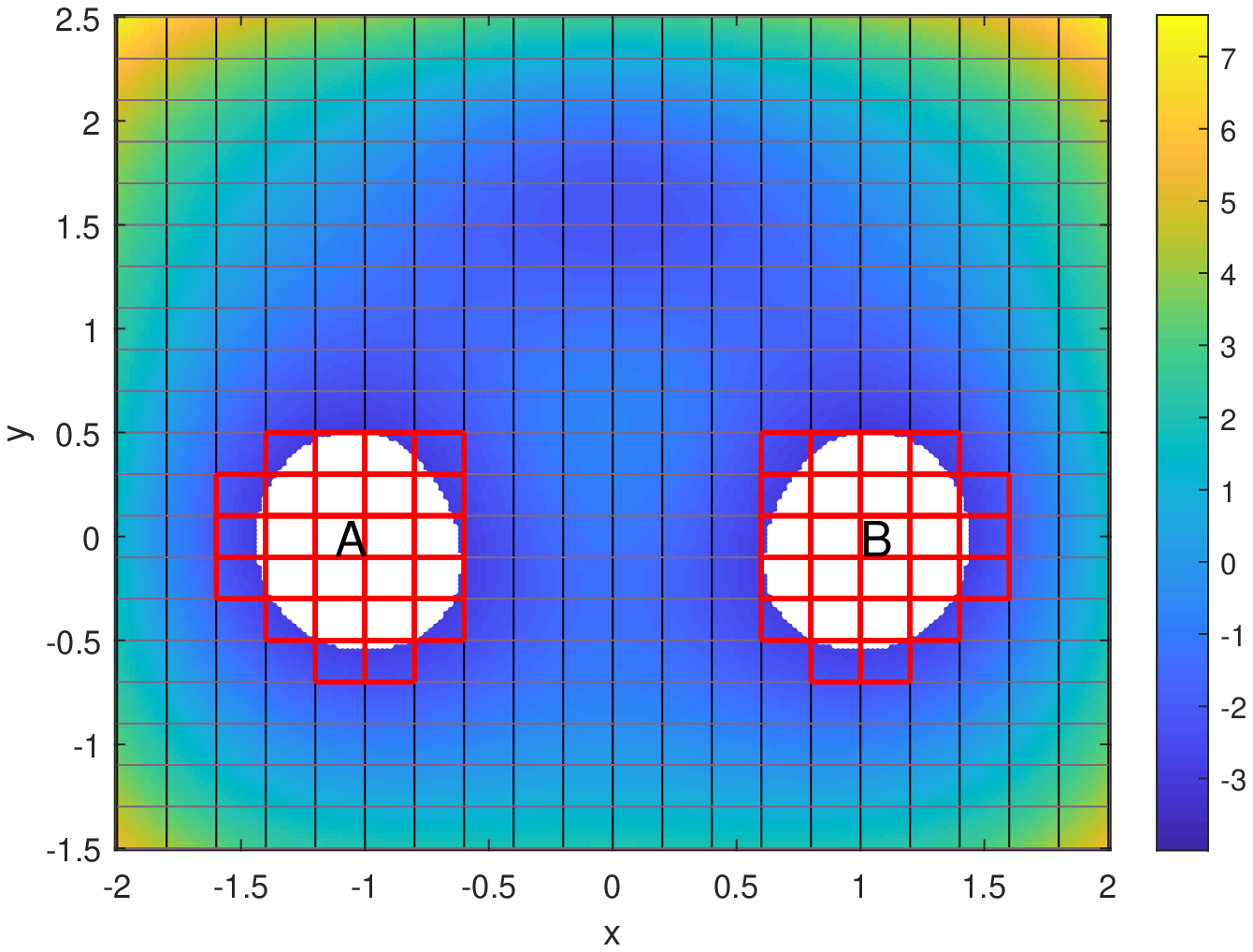}
        \caption{Representation of $A$ and $B$ (white) for a given $\{S_i\}_{i\in I}$ by cells in $J\cup K$ (red).}
    \label{figure:energy_landscape_top_view_mesh}
    \end{subfigure}
    \caption{The energy landscape $V$ and partition of state space.}
\end{figure*}

We simulated trajectories of $X$ with $\beta = 1.67$ and $\Gamma = \bb{I}$ in \eqref{eq:Smoluchowski} using the Euler-Maruyama method, with time step $\Delta t = 0.001$. This resulted in the scalar prefactor $\sqrt{2\Delta t\beta^{-1}}\approx 0.034$ for each standard normal random variable in the Euler-Maruyama method. 

For the results below, we chose $\mu$ to be the invariant Boltzmann-Gibbs measure with density proportional to $\exp(-\beta V(\cdot))$ on $S$, and used trapezoidal quadrature to compute the $\mu$-measure of cells.

\subsection{Approximate committor}
\label{ssec:numerical_results_approximate_committor}

To obtain the discrete committor function described in Section \ref{ssec:approximate_committor}, we sampled $10^4$ trajectories for each cell $S_i$ in the partition. For simplicity, we drew the initial condition of each trajectory independently from the uniform distribution on $S_i$. Since the uniform distribution on $S_i$ and the normalised restriction of the ergodic measure $\mu$ to $S_i$ in general do not coincide, this choice incurs an error; however, we shall show that for sufficiently small mesh size $h$ this error does not play a significant role relative to other errors. Each sample trajectory was terminated when the state of the trajectory reached $A$ or $B$. The approximate committor $\tilde{q}$ was computed from the computed trajectories using
\eqref{eq:approximate_committor_func}. 

To obtain the reference values of the committor $q$ for each mesh of side length $h$, we used that $q$ solves the Feynman-Kac or backward Kolmogorov PDE
\begin{align*}
 Lq(x)=&0,\quad x\in\inter{S\setminus (A\cup B)}
 \\
 q\vert_{\clos{A}}=&0,\quad 
 q\vert_{\clos{B}}=1, \quad
 \nabla q\vert_{\partial S}=0.
\end{align*}
where $L$ denotes the infinitesimal generator of \eqref{eq:Smoluchowski}; see \cite[Eq. (10)]{weinan2006towards}. We used the finite differences method to solve the PDE on each mesh, for the values of $h$ stated above. 

To compute the error $\Vert q-\tilde{q}\Vert_{L^2(\mu)}$, we computed the difference of the vector versions of $q$ and $\tilde{q}$, weighted by the vector of $\mu$-measures of each cell. We plot the results in Figure \ref{figure:eqbm_weighted_L2_error_committor} and observe that the observed values of the $L^2(\mu)$ error of the approximate committor decreased linearly with the partition width for the considered values of $\rho$, with an observed slope of approximately 1. Thus, the numerical results for the approximate committor are consistent with the bound in Corollary \ref{corollary:approximate_committor_convergence}.

\begin{figure*}[t!]
    \centering
    \begin{subfigure}[t]{0.475\textwidth}
        \includegraphics[width=1\linewidth]{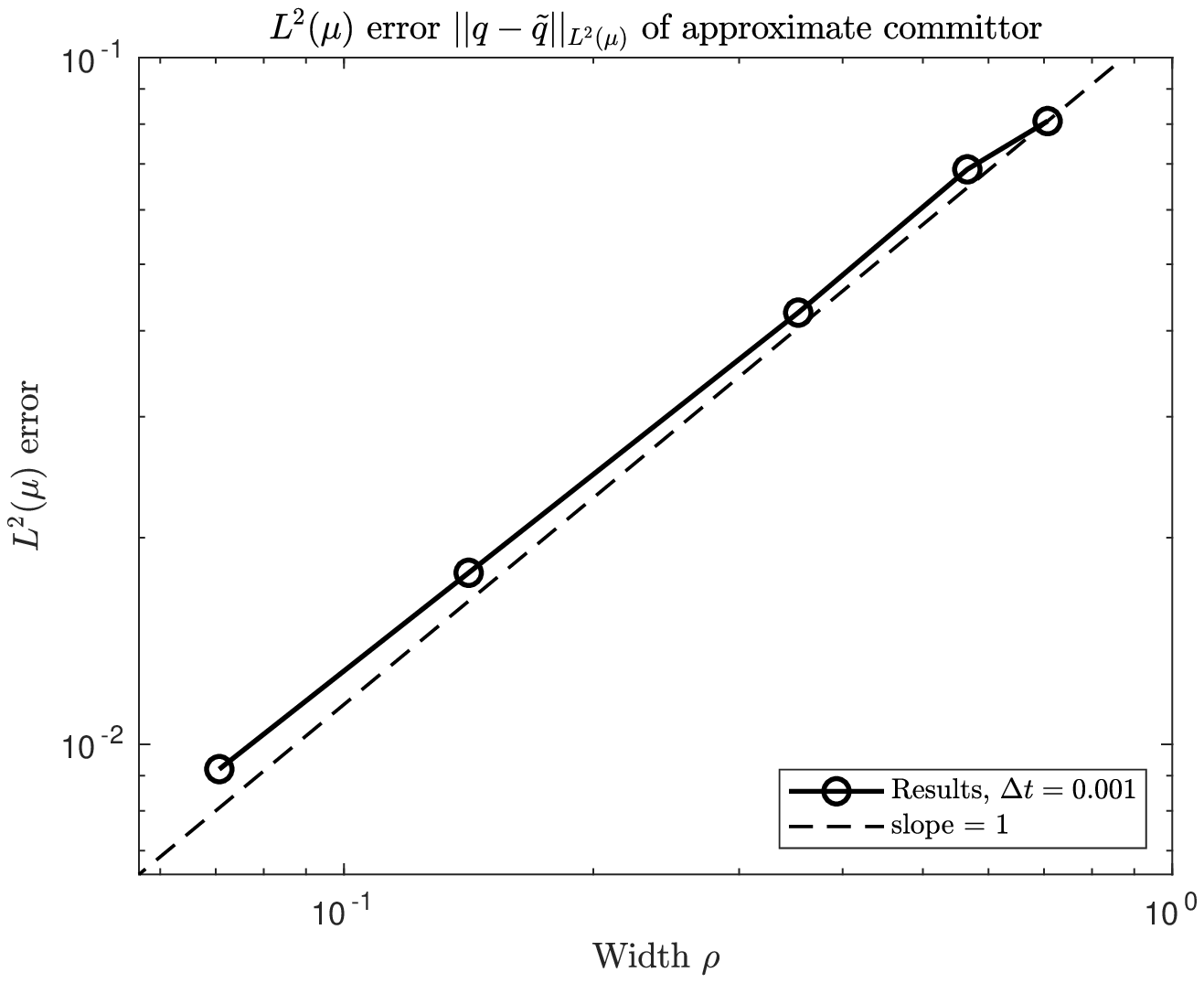}
        \caption{Plot of $\Vrt{q-\tilde{q}}_{L^2(\mu)}$ error.}
    \label{figure:eqbm_weighted_L2_error_committor}
    \end{subfigure}%
    \hskip2ex
    \begin{subfigure}[t]{0.475\textwidth}
        \centering
        \includegraphics[width=1\linewidth]{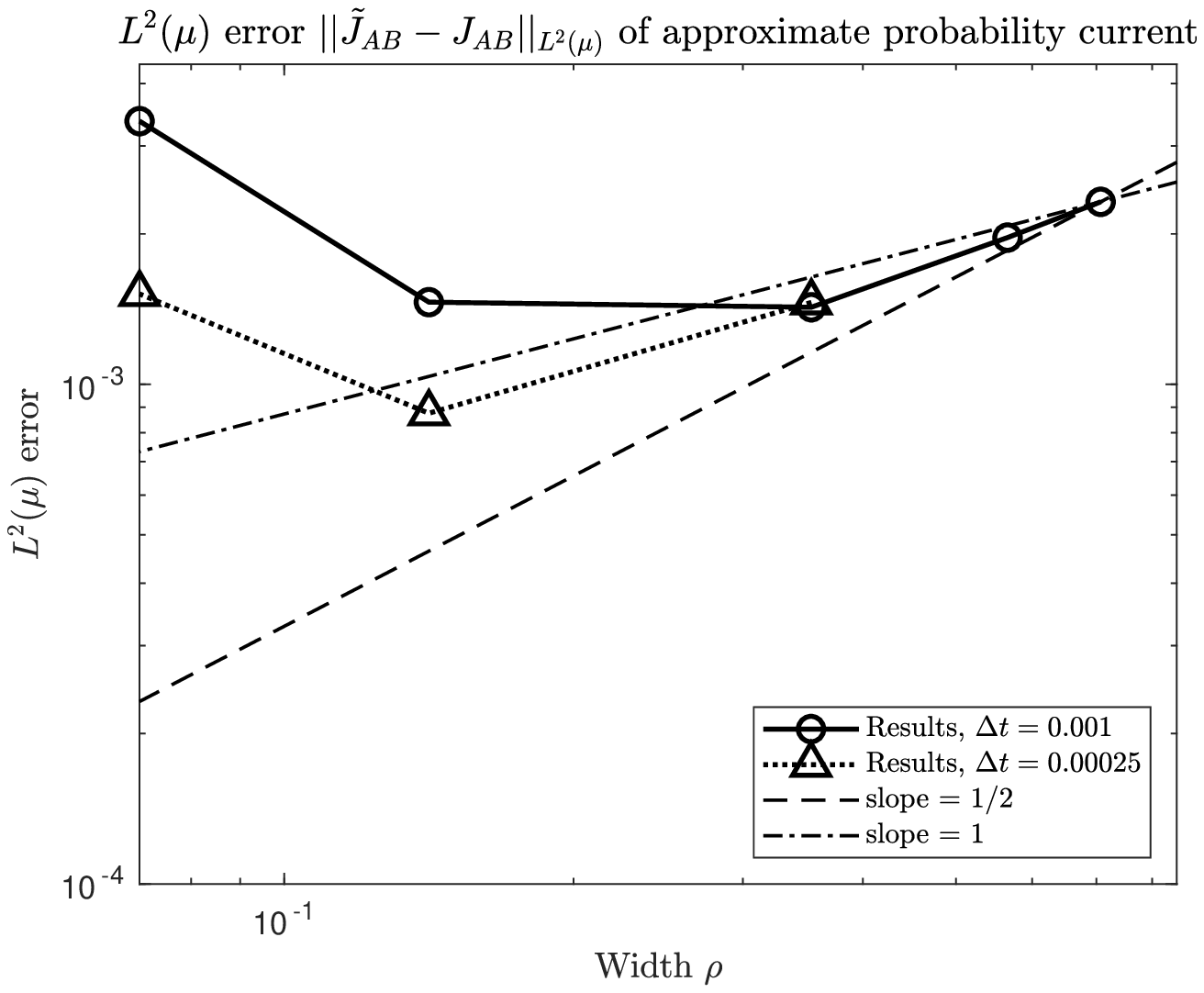}
        \caption{Plot of $\Vert J_{AB}-\tilde{J}_{AB}\Vert_{L^2(\mu)}$ error.}
    \label{figure:eqbm_weighted_L2_error_probcurrent}
    \end{subfigure}
    \caption{Plots of $L^2(\mu)$-errors of approximate committor and approximate probability current as a function of width $\rho$. See the discussion in \S\ref{ssec:numerical_results_approximate_committor} and \S\ref{ssec:numerical_results_approximate_probability_current}.}
\end{figure*}

\subsection{Approximate probability current}
\label{ssec:numerical_results_approximate_probability_current}

To compute the approximate probability current, we need to compute for each $i\in I$ the vector $\tilde{J}_{AB,i}\in\bb{R}^d$, which in turn requires the vector $\hat{\alpha}_i\in\mathbb{R}^{\card{\mathcal{N}(S_i)}}$ of surface area-normalised net flows of reactive trajectories; see \eqref{eq:def_JAB_i_tilde}. To do this, we followed \eqref{eq:alpha_ik_def} and sampled a long trajectory starting from the origin until $10^5$ reactive trajectory segments were collected. We then followed the procedure described in the paragraph immediately after Remark \ref{remark:approximate_probability_current_def_arbitrariness}. The duration of this long trajectory was computed as $T:=M\Delta t$, for $M$ the number of Euler-Maruyama steps. The observation time window parameter was set to be the Euler-Maruyama time step, i.e. $s:=\Delta t$. In particular, both limits were approximated by taking fixed values of $s$ and $T$. Using the reactive trajectory segments, we computed $\alpha_{ik}$ for each $i\in I$ and for each $k\in\mathcal{N}(S_i)$. Since each set $S_i$ in the partition is a square of side length $h$, it follows that $\sigma(\partial S_i\cap \partial S_k)=h$ for any adjacent $S_i$ and $S_k$. These sets of numbers then yielded the $\hat{\alpha}_{ik}$ by \eqref{eq:hat_alpha_ik}. For each $i\in I$, we solved the system \eqref{eq:normal_equation} for $\tilde{J}_{AB,i}$ by using the fact that $M_i=2I_2$, where $I_2$ is the $2\times 2$ identity matrix; see the paragraph immediately before \S\ref{ssec:approximate_streamlines}. We then constructed the approximate probability current according to \eqref{eq:approximate_probability_current_def}.

To obtain the reference values of the probability current $J_{AB}$ for each mesh of side length $h$, we used the fact that for diffusion processes given by Smoluchowski SDEs, the probability current is proportional to the product of the Boltzmann-Gibbs density with the gradient of the committor
\begin{equation*}
 J_{AB}(x)=\frac{1}{Z}e^{-\beta V(x)}\beta^{-1}\nabla q(x),
\end{equation*}
see \cite[Eq. (6)]{TPT_illustration}. For each mesh of side length $h$, we used the finite differences approximation of $q$ to approximate the gradient $\nabla q$ and computed the reference values of the probability current using the equation above.

We plot the $L^2(\mu)$ error of the approximate probability current relative to the finite differences solution in Figure \ref{figure:eqbm_weighted_L2_error_probcurrent}. For the three largest values of $h$, i.e. for $h\in\{0.5, 0.4,0.25\}$, the observed values of the $L^2(\mu)$ error of the approximate committor decrease linearly with the partition width; the slope of the line is between $1/2$ and $1$. Subsequently, the observed values of the $L^2(\mu)$ error increase for $h\in\{0.1,0.05\}$. These observed results are not consistent with the convergence behaviour described by Corollary \ref{corollary:approximate_probability_current_convergence}. 

The increase in $L^2(\mu)$ error for $h\in\{0.1,0.05\}$ is a consequence of under-counting of transitions, which in turn is due to an observation time window that is too large for the mesh side length parameter $h$. For $\Delta t=0.001$, recall that $\sqrt{2\Delta t\beta^{-1}}\approx 0.034$. Hence, with more than 95\% probability, the magnitude of each random perturbation in the Euler-Maruyama method is less than or equal to $0.1$. As a result, we expect to significantly underestimate the number of crossings of $\partial S_i\cap\partial S_k$ when each $S_i$ has side length approximately equal to or less than $\sqrt{2\Delta t\beta^{-1}}$, i.e. for $h\in\{0.05,0.1\}$.

To check the validity of the preceding explanation for the increase in $L^2(\mu)$ errors for $h\in\{0.1,0.05\}$, we sampled for each $h\in\{0.25,0.1,0.05\}$ a single long trajectory with $\Delta t=0.00025$ and computed the corresponding approximate probability current. For $\Delta t=0.00025$, $\sqrt{2\Delta t\beta^{-1}}\approx 0.017$ and thus with more than 95\% probability the magnitude of each random perturbation in the Euler-Maruyama method is less than or equal to $0.051$. The results of the additional experiments are plotted in Figure \ref{figure:eqbm_weighted_L2_error_probcurrent}: for $h=0.1$ and $\Delta t=0.00025$ the result approximately follows the same linear trend of the observed $L^2(\mu)$ error for $h\in\{0.5,0.4,0.25\}$ and $\Delta t=0.001$. For $h=0.05$, $\sqrt{2\Delta t\beta^{-1}}\approx h$, and we observe that the $L^2(\mu)$ error deviates significantly from this trend. These results support our proposed explanation. 

We further analysed the error between the approximate probability current $\tilde{J}_{AB}$ and the probability current $J_{AB}$ obtained from finite differences by computing two error metrics at each point: the scaling ratio $R(x)$ and the direction error $D(x)$,
\begin{equation}
\label{eq:discrepancy_metrics_probability_current}
 R(x):=\frac{\vert \tilde{J}_{AB}(x)\vert_2}{\abs{J_{AB}(x)}_2},\quad D(x):=\textup{arccos}\left( \frac{\tilde{J}_{AB}(x)^\top J_{AB}(x)}{\vert \tilde{J}_{AB}(x)\vert_2\abs{J_{AB}(x)}_2}\right),
\end{equation}
where we choose units so that $0\leq D(x)\leq \pi$. For $\Delta t=0.00025$ and for each value of $h\in\{0.05,0.1,0.25\}$, we computed the scaling ratio and direction error for cells $S_i$ on which the value of the potential $V$ is less than or equal to 1. We restrict to these cells because the sampled reactive trajectories spend significantly less time in the region $[V>1]:=\{x\in S\ :\ V(x)>1\}$ compared to its complement $[V\leq 1]:=\{x\in S\ :\ V(x)\leq 1\}$. As a result, we do not expect the observed statistics for cells in $[V>1]$ to provide reasonable approximations of their true values in $[V>1]$, compared to the observed statistics for cells in $[V\leq 1]$. In addition, we thresholded the values of the scaling ratio to the interval $0\leq R(x)\leq 2$ to reduce the influence of large outliers and to make the comparison symmetric about the desired scaling ratio of 1. We plot the frequency histograms for the direction error $D$ and scaling ratio $R$ in Figures \ref{figure:direction_error_frequency_histogram} and \ref{figure:scaling_ratio_frequency_histogram}.

\begin{figure*}[t!]
    \centering
    \begin{subfigure}[t]{0.475\textwidth}
        \includegraphics[width=1\linewidth]{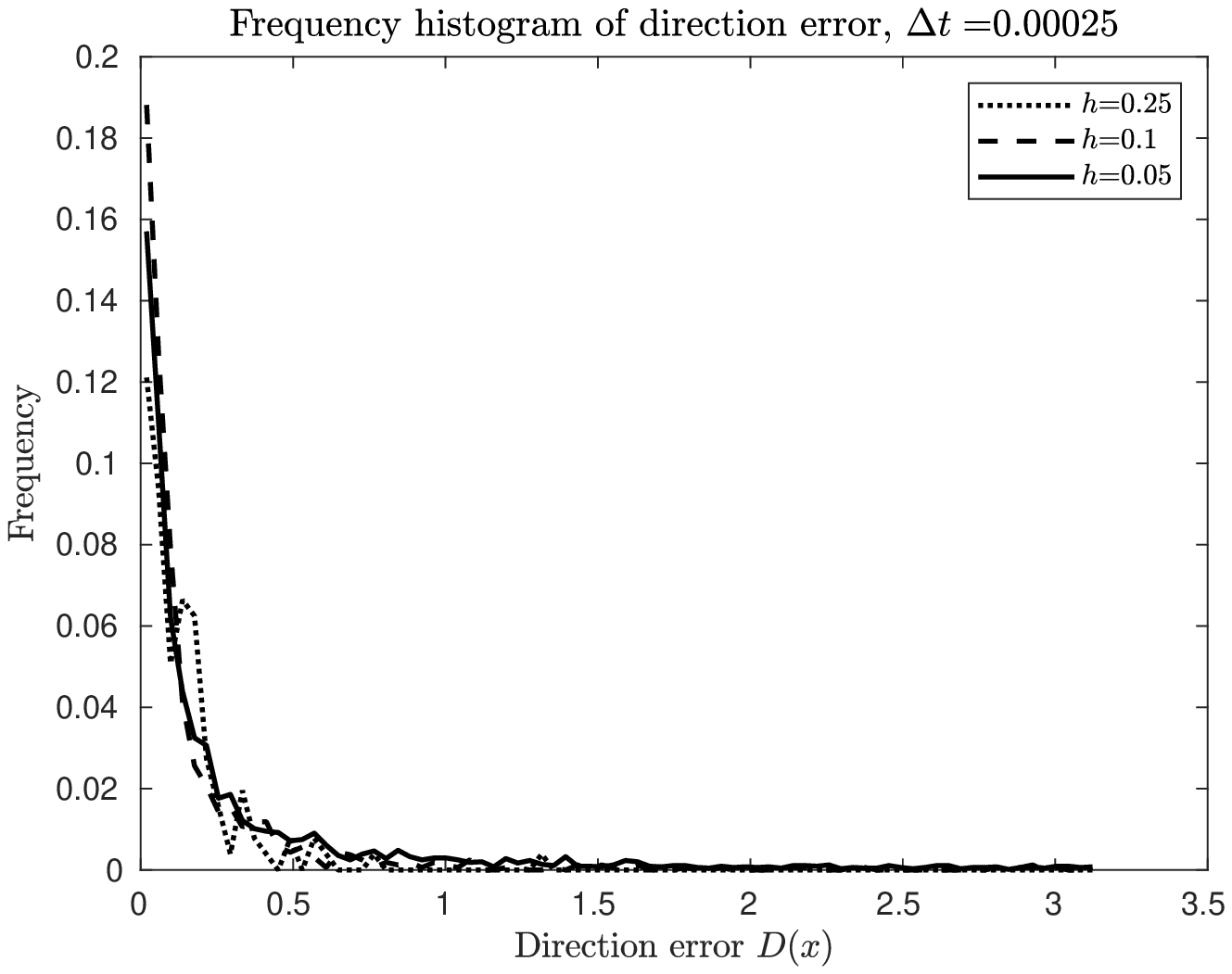}
        \caption{Frequency histogram of $D(x)$.}
    \label{figure:direction_error_frequency_histogram}
    \end{subfigure}%
    \hskip2ex
    \begin{subfigure}[t]{0.475\textwidth}
        \centering
        \includegraphics[width=1\linewidth]{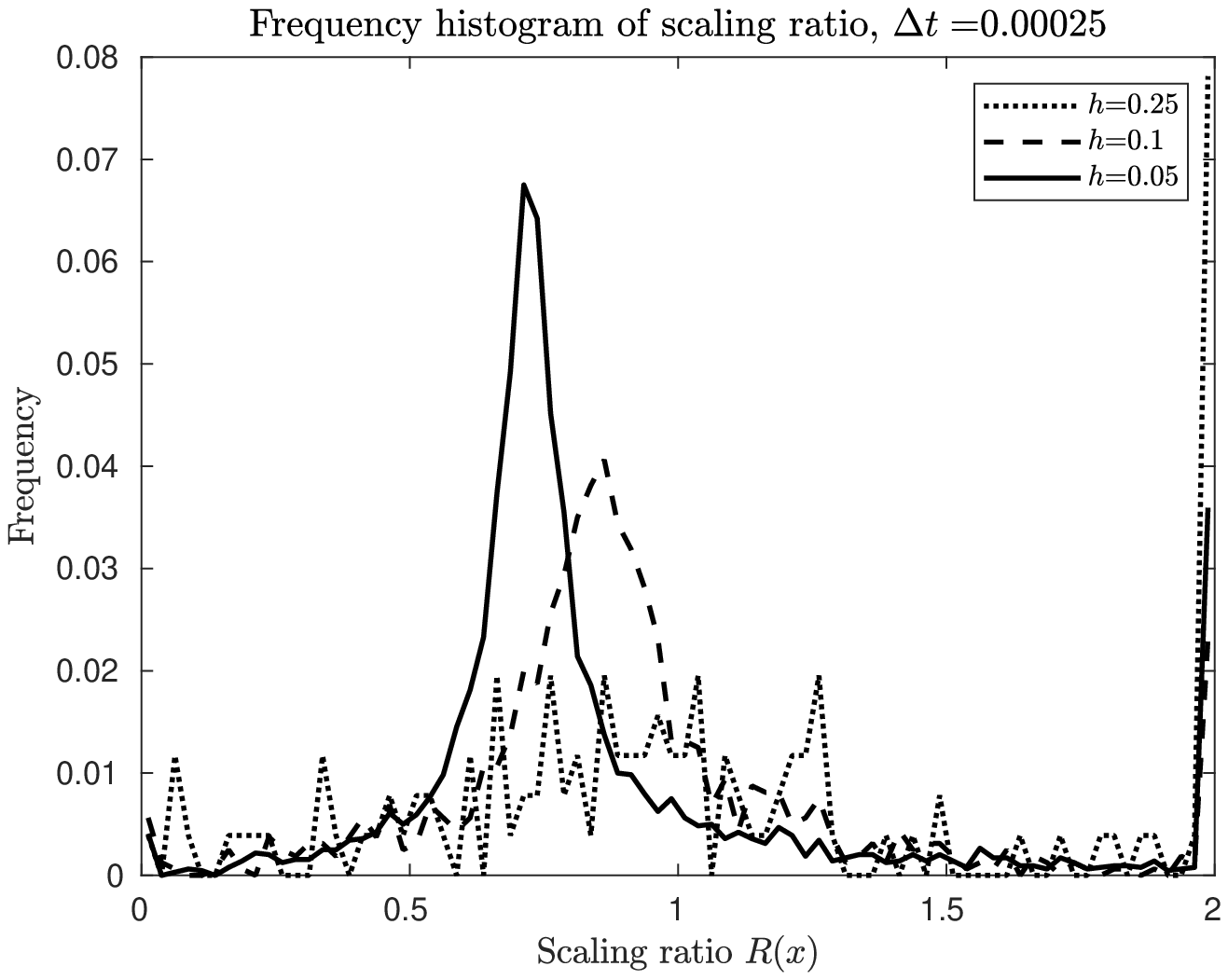}
        \caption{Frequency histogram of $R(x)$.}
    \label{figure:scaling_ratio_frequency_histogram}
    \end{subfigure}
    \caption{Frequency histograms of direction error $D(x)$ and scaling ratio $R(x)$ for $\Delta t=0.00025$ and $h\in\{0.05,0.1,0.25\}$.}
\end{figure*}

In Figure \ref{figure:direction_error_frequency_histogram} we observe that the frequency histograms of the direction error $D(x)$ are approximately the same for $h\in\{0.05,0.1,0.25\}$. In contrast, Figure \ref{figure:scaling_ratio_frequency_histogram} shows that the frequency histograms for the scaling ratio $R(x)$ differ in some aspects for varying $h$: as $h$ decreases from $h=0.25$ to $h=0.05$, the frequency histogram of the scaling ratio becomes more sharply peaked, and in addition the observed mean of the frequency histogram decreases away from 1. In addition, for $h=0.25$, the frequency of large outliers is at least twice the frequency for $h=0.1$ and $h=0.05$. 

Together, Figures \ref{figure:eqbm_weighted_L2_error_probcurrent}, \ref{figure:direction_error_frequency_histogram}, and \ref{figure:scaling_ratio_frequency_histogram} suggest that, relative to the direction error $D(x)$, the observed distribution of the scaling ratio $R(x)$ of the approximate probability current plays a more important role in explaining the discrepancy between the observed and predicted slopes of the numerical $L^2(\mu)$ error of the approximate probability current. To support this claim, we plot the direction error $D(x)$ and the absolute deviation $\vert R(x)-1\vert$ of the scaling ratio for the numerical results obtained with the parameters $h=0.25$ and $\Delta t=0.00025$ in Figures \ref{figure:direction_error_plot} and \ref{figure:deviation_scaling_ratio_plot} respectively. 

\begin{figure*}[!ht]
    \centering
    \begin{subfigure}[t]{0.475\textwidth}
        \includegraphics[width=1\linewidth]{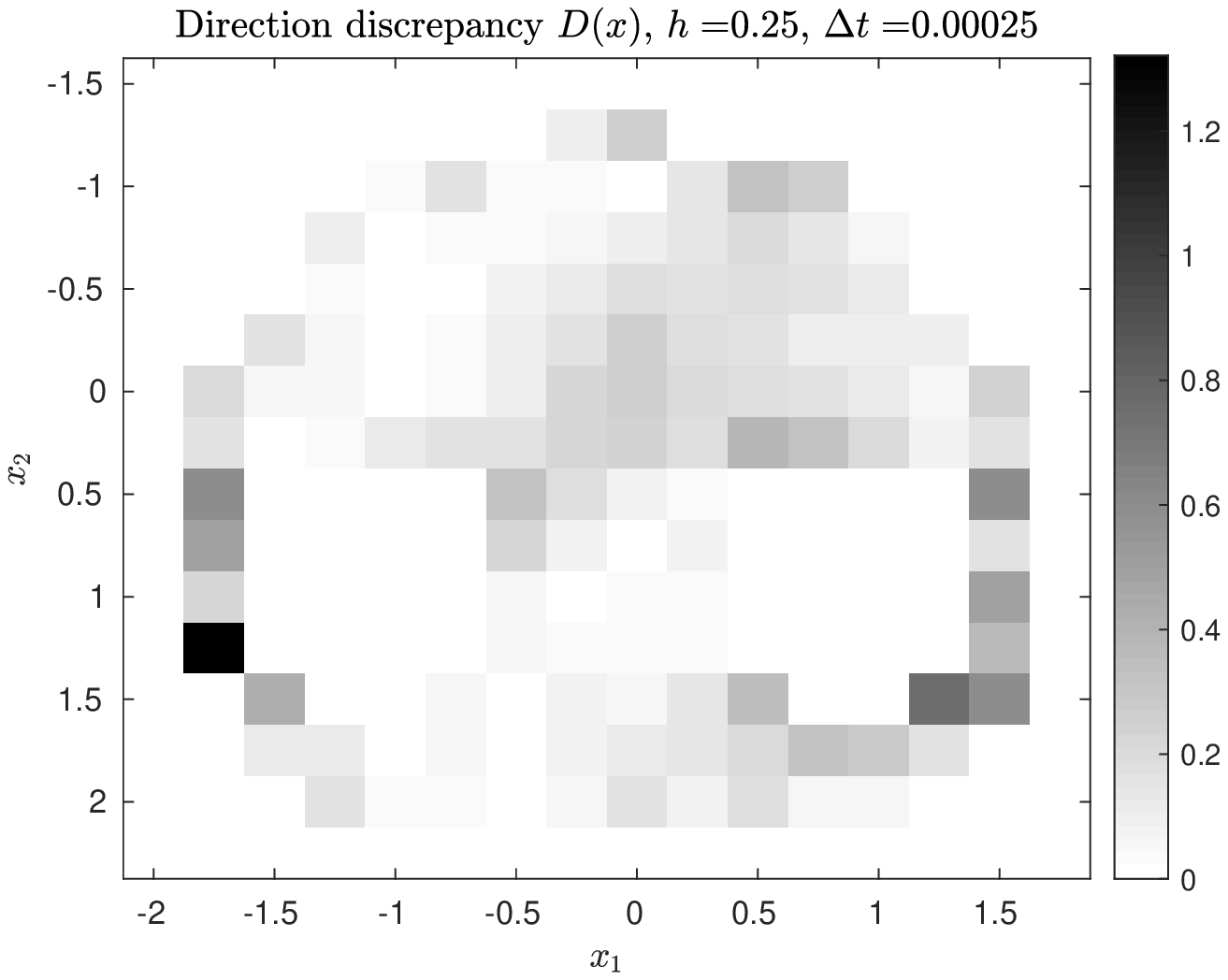}
        \caption{Direction error $D(x)$.}
    \label{figure:direction_error_plot}
    \end{subfigure}%
    \hskip2ex
    \begin{subfigure}[t]{0.475\textwidth}
        \centering
        \includegraphics[width=1\linewidth]{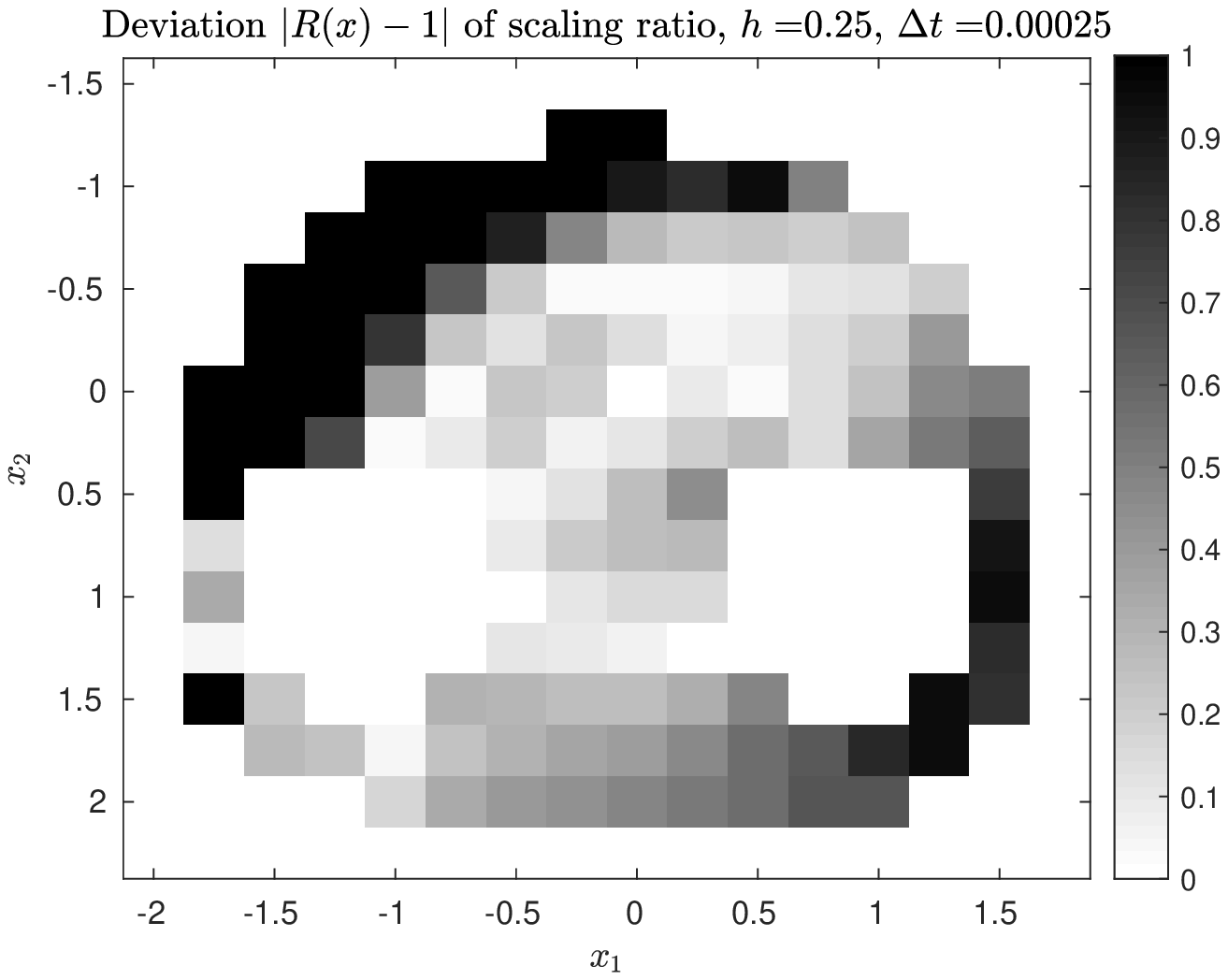}
        \caption{Deviation $\vert R(x)-1\vert$ of scaling ratio.}
    \label{figure:deviation_scaling_ratio_plot}
    \end{subfigure}
    \caption{Plots of direction error $D(x)$ and deviation $\vert R(x)-1\vert$ of scaling ratio for $h=0.25$ and $\Delta t=0.00025$. See \eqref{eq:discrepancy_metrics_probability_current} for definitions of $D(x)$ and $R(x)$.}
\end{figure*}

In Figure \ref{figure:direction_error_plot} we observe that the direction error is less than 0.5 over much of the region of interest. In contrast, in Figure \ref{figure:deviation_scaling_ratio_plot} the absolute deviation $\vert R(x)-1\vert$ is more than 0.9 over much of the region of interest. These figures suggest that, given enough reactive trajectory segments, our method can identify the direction of the probability current $J_{AB}$ better than the magnitude of the probability current. The errors in the estimation of the probability current tend to be larger over areas that the reactive trajectories visit less often, although there are exceptions. Remarkably, the plot analogous to Figure \ref{figure:deviation_scaling_ratio_plot} for $\Delta t=0.001$ (not shown) is very similar to the plot for $\Delta t=0.00025$ shown above; this suggests that the observation time window $s$ does not play a significant role in the observed deviations in scaling ratio. These observations suggest that too short a trajectory duration $T$, or equivalently, insufficient sampling of reactive trajectories explains the discrepancy between the observed slope and the predicted slope of the numerical $L^2(\mu)$ errors in the approximate probability current.

\section{Conclusion}
\label{sec:conclusion}

In this paper, we presented a method for approximating key objects from TPT for ergodic diffusion processes, namely the committor, probability current, and streamlines, using a Voronoi tessellation of state space and sample trajectory data. In the case of the approximate committor, only sample trajectory data are needed, and we can interpret the approximate committor as the committor of a non-Markovian jump process on the Delaunay graph $G$. For the approximate probability current and the approximate streamlines, sample \emph{reactive} trajectory data are needed, and we cannot interpret either object as a discrete analogue on $G$ of the corresponding object from TPT for diffusion processes. 

For each approximate object, we proved a bound on the approximation error, while neglecting for simplicity both statistical error due to finite sample sizes and errors due to the numerical method used for generating trajectories of the diffusion process. We used these bounds to prove that the approximate objects converge to the continuous objects from TPT in the continuum limit. We presented numerical results for the approximate committor that are consistent with our theoretical prediction of convergence behaviour. We presented numerical results for the approximate probability current, and explained their deviations from the predicted values in terms of the approximation of the $s\to 0^+$ and $T\to\infty$ limits in \eqref{eq:alpha_ik_def} with fixed values of $s$ and $T$.

The approach we have taken here is statistical and data-driven in nature, in the sense that our method relies only on computing observed statistics from sampled reactive trajectories and the Voronoi tessellation. This may be a disadvantage when reactive trajectory data is difficult to obtain. Furthermore, the numerical results in \S\ref{ssec:numerical_results_approximate_probability_current} showed that statistical error or under-sampling in regions of low probability can lead to slow convergence. On the other hand, the data-driven feature may be advantageous when the task of obtaining trajectory data is easier compared to the task of learning the drift and diffusion coefficients of the SDE.

Our method differs from methods for TPT that involve constructing a discrete state-space Markov process for approximating the underlying diffusion process. This feature comes at the cost of interpretability: the approximate probability current and streamlines cannot be interpreted as discrete TPT objects on the Delaunay graph associated to the Voronoi tessellation. On the other hand, we do not need to ensure that an approximating process obtained from discretisation of a continuous state space possesses the Markov property. In particular, we do not need to identify a suitable lag time, and can therefore avoid the preprocessing steps and heuristic reasoning that are often used to identify the lag time. In addition, the error and convergence analysis of the approximation error is simpler, because the definitions of the approximate objects of our method are not based on preserving certain properties on a graph, but on the continuous objects from TPT for diffusion processes.

We now mention some aspects that should be addressed in future work in order to improve the applicability of our method. 

In order to compute the approximate probability current, one needs to compute the $(d-1)$-dimensional volumes of the facets between adjacent Voronoi cells; see \eqref{eq:hat_alpha_ik}. Computing these volumes requires computing the full Voronoi tessellation, which will be expensive in high dimensions. One possible approach to circumvent this would be to use Voronoi tessellations in which every Voronoi cell is a rigid transformation of one Voronoi cell. However, this would prevent the user from incorporating a priori information about the SDE such as the energy landscape when placing generators for the tessellation. Thus, it would be desirable to modify the approach so that it is not necessary to compute the volume of every facet between adjacent cells. A possible approach would be to avoid a full tessellation of state space, e.g. by using core sets.

In \S\ref{ssec:numerical_results_approximate_probability_current}, we observed that the simple Monte Carlo approach that we used above could accurately estimate the committor probability, but not the probability current, over all regions in state space. One important reason for this is that it is easier to obtain the trajectory data needed for estimating committor probabilities than to obtain reactive trajectory data for estimating net fluxes: for the former, one starts a trajectory from a point or region of choice and waits until the trajectory reaches $A$ or $B$; for the latter, one starts a reactive trajectory segment on $\partial A$ and waits until it reaches $\partial B$ without re-entering $A$. Another important reason is that to construct the approximate probability current, one needs to estimate net fluxes across all facets. The accuracy of the net flux estimates depends strongly on how well the sampled reactive trajectories explore state space: the net flux estimates in regions that are explored less will be less accurate. It would be desirable to modify the approach so that sampling reactive trajectories is easier, and so that the sample reactive trajectories explore state space more uniformly. This could be achieved by using enhanced sampling techniques, for example. 

For applications where trajectory data is difficult to obtain, it would be of interest to extend the method so that it could be applied to sparse trajectory data, where `sparse' means that subsequent measurement times are far apart from each other, or where the trajectory data degenerates to a cloud of point data. In full generality, such a task appears intractable. However, under certain assumptions, e.g. the assumption that the points lie on an energy landscape that determines the dynamics, it might be possible to develop such an extension. For example, a recent analysis \cite{heida} has shown that applying the so-called square root approximation \cite{lie_voronoi} to generate an approximate transition rate matrix from point cloud data on an energy landscape produces an approximation of the associated infinitesimal generator. It would be of interest to modify the definition of the approximate probability current, in order to make our approach applicable to such data.

\section*{Acknowledgments}
NC and TC were supported by project grant CH14 of the Einstein Center for Mathematics Berlin (ECMath), and by the DFG Research Center Matheon ``\emph{Mathematics for key technologies}'' in Berlin. TC is supported by the German Ministry of Research and Education (BMBF) project grant 3FO18501 (Forschungscampus MODAL). The research of NC and HCL has been partially funded by Deutsche Forschungsgemeinschaft (DFG) - SFB1294/1 - 318763901.

The authors thank Christof Sch\"{u}tte and Gabriel Stoltz for their helpful feedback on early versions of this work. 

\bibliographystyle{plain}
\bibliography{cvgt_disc_tpt}

\appendix
\section{Auxiliary results and proofs}
\label{sec:appendix}

\subsection{Approximate committor}
\label{ssec:appendix_committor}

The purpose of this section is to prove Proposition \ref{proposition_committors_equal}, i.e. that for a given tessellation $\{S_i\}_{i\in I}$, $\hat{q}_i = \tilde{q}_i$ for all $i\in I$. We accomplish this by using the notion of a regular conditional distribution. To define a regular conditional distribution, we first recall some definitions; see \cite[Section 8.3]{klenke2008probability}.
\begin{definition}[Stochastic kernel]
\label{definition:stochastic_kernel}
	Let $(\Omega_1,\mathcal{A}_1)$ and $(\Omega_2,\mathcal{A}_2)$ be measurable spaces. A map $\kappa: \Omega_1\times \mathcal{A}_2 \rightarrow \left[0,\infty\right]$ is called a \emph{stochastic kernel} from $\Omega_1$ to $\Omega_2$ if\begin{itemize}
		\item [(i)] $\kappa(\cdot,A_2)$ is $\mathcal{A}_1$-measurable for any $A_2\in \mathcal{A}_2$, and
		\item [(ii)] $\kappa(\omega_1,\cdot)$ is a $\sigma-$finite probability measure on $(\Omega_2,\mathcal{A}_2)$ for any $\omega_1 \in \Omega_1$.
	\end{itemize}
\end{definition}

\begin{definition}[Regular conditional distribution]
\label{definition:RCD}
	Let $(\Omega_1,\mathcal{A}_1,\mathbb{P})$ be a probability space, $(E,\mathcal{E})$ be a measurable space, and $Y$ be an $E$-valued random variable on $(\Omega_1,\mathcal{A}_1,\mathbb{P})$. Let $\mathcal{F}\subset \mathcal{A}_1$ be a $\sigma$-algebra on $\Omega_1$. A stochastic kernel $\kappa_{Y,\mathcal{F}}$ from $(\Omega_1,\mathcal{F})$ to $(E,\mathcal{E})$ is called a \emph{regular conditional distribution of $Y$ given $\mathcal{F}$} if 
	\begin{equation}\label{RCD_condition}
		\kappa_{Y,\mathcal{F}} (\omega,C) = \mathbb{P}\left(Y\in C \vert \mathcal{F}\right) (\omega)
	\end{equation} for $\mathbb{P}-$almost all $\omega \in \Omega_1$ and for all $C\in \mathcal{E}$. 
	
	If $\mathcal{F}$ is generated by a random variable $X$ defined on $(\Omega_1,\mathcal{A}_1,\mathbb{P})$ that takes values in some measurable space $(E',\mathcal{E}')$ then the stochastic kernel $\kappa_{Y,X}$ from $(E',\mathcal{E}')$ to $(E,\mathcal{E})$ that is defined by
	\begin{equation}
	\label{def_RCD_of_Y_given_X}
	\kappa_{Y,X}(x,C) = \mathbb{P}\left(Y\in C  \vert X = x \right) = \kappa_{Y,\sigma(X)}(\omega,C),\quad \omega\in X^{-1}(\{x\})
	\end{equation}
	 is called a \emph{regular conditional distribution of $Y$ given $X$}.
\end{definition}
The existence of the regular conditional distribution of $Y$ given $\mathcal{F}$ for a random variable $Y$ taking values in a Borel space is given in \cite[Theorem 8.37]{klenke2008probability}.

\begin{lemma}[Committors are regular conditional probabilities]\label{lemma:committors_are_rcp}
The committor function $q$ defined in \eqref{eq:forward_committor} is a regular conditional probability.
\end{lemma}
\begin{proof}
    Set the random variables $X,Y, $ and the set $C$ in \eqref{def_RCD_of_Y_given_X} to be equal  $X_0$, $X_{\tau_{A\cup B}(X)}$, and $B$, respectively. Then 
    $$\kappa_{X_{\tau_{A\cup B}(X)}, X_0}(x,B) = \bb{P}(X_{\tau_{A\cup B}(X)} \in B \vert X_0 \in x) = q(x),$$
    where we used the definition \eqref{eq:forward_committor} of the committor function in the second equation.
\end{proof}

The following result is reproduced from \cite[Theorem 8.38]{klenke2008probability}.
\begin{theorem}
	\label{theorem_cond_expectation_with_RCD}
	Let $Y$ be a random variable on $(\Omega_1, \mathcal{A}_1,\mathbb{P})$ with values in some set $E$, and equip $E$ with the Borel $\sigma$-algebra $\mathcal{E}$. Let $\mathcal{F}\subset \mathcal{A}_1$ be a $\sigma$-algebra and let $\kappa_{Y,\mathcal{F}}$ be a regular conditional distribution of $Y$ given $\mathcal{F}$. Further, let $f:E\rightarrow \mathbb{R}$ be measurable and $\mathbb{E}\left[|f(Y)| \right]<\infty.$ Then 
	\begin{equation*} 
	 \mathbb{E}\left[f(Y)\vert\mathcal{F} \right](\omega) = \int_{E} f(y)\kappa_{Y,\mathcal{F}}(\omega,\mathrm{d}y)
	\end{equation*}
$\mathbb{P}$-almost surely.
\end{theorem}

We use Theorem \ref{theorem_cond_expectation_with_RCD} to prove the following lemma, which expresses the joint probability of a pair of random variables $X$ and $Y$ in terms of the regular conditional probability of $Y$ given $X$.
\begin{lemma}
\label{lemma:RCD_for_committors}
	Let $X$ and $Y$ be random variables on $(\Omega_1, \mathcal{A}_1,\mathbb{P})$, taking values in the measurable spaces $(E',\mathcal{E}')$ and $(E,\mathcal{E})$ respectively. Then for any $C\in \mathcal{E}$ and $D\in\mathcal{E}'$,
	\begin{equation*}
	%\label{definition:RCD_for_committors}
	\mathbb{P}\left(Y\in C,X\in D\right) = \int_{D}
	\kappa_{Y,X}
	(x,C)\mathbb{P}\circ X^{-1}(\mathrm{d}x).
	\end{equation*}
\end{lemma}
\begin{proof}
Let $\sigma(X)$ denote the $\sigma$-algebra generated by $X$. Then
\begin{align}
    \mathbb{P}(Y\in C, X\in D) & = \mathbb{E}[\mathbf{1}_D(X)\mathbf{1}_C(Y)] =\mathbb{E}\left[\mathbf{1}_D(X)\mathbb{E}[\mathbf{1}_C(Y)\vert \sigma(X)] \right] .
    \label{eq:RCD_for_committors_4th}
\end{align}
Observe that by Theorem \ref{theorem_cond_expectation_with_RCD}, \eqref{def_RCD_of_Y_given_X}, and Definition \ref{definition:stochastic_kernel},
\begin{align}\label{eq:RCD_for_committors_2nd}
\begin{split}
    \mathbb{E}\left[\mathbf{1}_C(Y)\vert \sigma(X) \right](\omega) &= \int_E \mathbf{1}_C(y)\kappa_{Y,\sigma(X)}(\omega,\mathrm{d}y) = \int_E \mathbf{1}_C(y)\kappa_{Y,X}(X(\omega),\mathrm{d}y)\\
    &= \kappa_{Y,X}(X(\omega),C).
    \end{split}
\end{align}
Therefore,
\begin{align*}
    \mathbb{P}(Y\in C, X\in D) &=  \int_{\Omega}\mathbf{1}_D \left(X(\omega) \right) \kappa_{X,Y}\left(X(\omega),C\right) \mathbb{P}(\mathrm{d}\omega) \\
    & = \int_{E}\mathbf{1}_D(x)\kappa_{X,Y}(x,C)\mathbb{P}\circ X^{-1}(\mathrm{d}x),
\end{align*}
where we used \eqref{eq:RCD_for_committors_4th} and \eqref{eq:RCD_for_committors_2nd} in the first equation and the change of variables formula in the second equation. This completes the proof.
 \end{proof}

We now use Lemma \ref{lemma:RCD_for_committors} to prove Proposition \ref{proposition_committors_equal}, which states that for every $i\in I$, the quantities $\hat{q}_i$ and $\tilde{q}_i$ defined in \eqref{proj_committor} and \eqref{eq:approximate_committor_ith_component} are equal, under the assumption that $\bb{P}\circ (X_0)^{-1}=\mu$.

\begin{proof}[Proof of Proposition \ref{proposition_committors_equal}]
	Let $i\in I$ be arbitrary. Recall from \eqref{proj_committor} that 
	\begin{equation*}
	\hat{q}_i=\frac{1}{\mu(S_i)}\langle q,\mathbf{1}_{S_i} \rangle_{\mu} = \frac{1}{\mu(S_i)} \int_{S}q(x)\mathbf{1}_{S_i}(x)\mu(\mathrm{d}x).
	\end{equation*}
	The definition \eqref{eq:approximate_committor_ith_component}, the construction of $Y$ and Lemma \ref{lemma:tau_AB_equals_tau_JK}, and the hypothesis that $\bb{P}\circ (X_0)^{-1}=\mu$ imply that 
	\begin{align*}
	\tilde{q}_i &= \frac{\mathbb{P}\left( Y_{\tau_{J\cup K}(Y)}\in K,\  Y_0= i \right)}{\mathbb{P}\left( Y_0 =i\right)} = \frac{\mathbb{P}\left( X_{\tau_{A\cup B}(X)}\in B,\  X_0 \in S_i \right)}{\bb{P}(X_0 \in S_i)}
	\\ &= \frac{\mathbb{P}\left( X_{\tau_{A\cup B}(X)}\in B,\  X_0 \in S_i \right)}{\mu(S_i)}.
	\end{align*}
	Thus, it suffices to show that 
	\begin{equation*}
	 \int_{S}q(x)\mathbf{1}_{S_i}(x)\mu(\mathrm{d}x)=\mathbb{P}\left( X_{\tau_{A\cup B}(X)}\in B,\  X_0 \in S_i \right).
	\end{equation*}
	By Lemma \ref{lemma:committors_are_rcp}, the left-hand side can be rewritten as
	\begin{align*}
	 \int_{S}q(x)\mathbf{1}_{S_i}(x)\mu(\mathrm{d}x)&=\int_{S_i}\kappa_{X_{\tau_{A\cup B}(X)},X_0}(x,B)\mu(\mathrm{d}x).
	\end{align*}
	Using that $\mu=\mathbb{P}\circ (X_0)^{-1}$ and Lemma \ref{lemma:RCD_for_committors}, we obtain
	\begin{align*}
	 \int_{S_i}\kappa_{X_{\tau_{A\cup B}(X)},X_0}(x,B)\mu(\mathrm{d}x) & =\int_{S_i}\kappa_{X_{\tau_{A\cup B}(X)},X_0}(x,B)\mathbb{P}\circ X_{0}^{-1}(\mathrm{d}x)\\ &=\mathbb{P}(X_{\tau_{A\cup B}(X)}\in B,X_0\in S_i),
	\end{align*}
	yielding the desired conclusion.
\end{proof}

\subsection{Approximate probability current}
\label{ssec:appendix_probability_current}

In this section, we consider $d\in\bb{N}$, $d\geq 2$. By `$d$-polytope', we mean a bounded convex polytope $P$ with $\dim{P}=d$.

\begin{lemma}
\label{lemma:d_polytopes_have_dplus1_facets}
	If $P\subset\bb{R}^d$ is a $d$-polytope, then $P$ at least $d+1$ facets.
\end{lemma}
\begin{proof}
Any such polytope admits a decomposition into $d$-simplices. Since every $d$-simplex has $d+1$ facets, the conclusion follows.
\end{proof}

\begin{corollary}
\label{corollary:rank_of_any_cell_is_d}
      If $P$ is a $d$-polytope in $\bb{R}^d$, then $P$ has at least $d$ linearly independent outer normals. In particular, the matrix $N$ whose rows are the outer unit normals to $P$ has full rank.
\end{corollary}
\begin{proof}
The last statement follows from the first, so it suffices to prove the first statement. By Lemma \ref{lemma:d_polytopes_have_dplus1_facets}, every $d$-polytope has at least $d+1$ outer normals. We will prove by contradiction that there exist $d$ linearly independent outer normals. Suppose $P$ has no more than $d-1$ linearly independent outer normals. Then the normals to the facets of $P$ span at most a $(d-1)$-dimensional space, which implies that there exists a hyperplane $H$ in $\mathbb{R}^d$ containing all the outer normals of $P$. Let $v$ be normal to $H$, and let $n$ be an arbitrary outer normal associated to to some facet $F$ of $P$. Then $v$ and $n$ are orthogonal, which implies that $v$ is parallel to $F$, and thus that $F$ is unbounded along the direction of $v$. This implies that $P$ is unbounded, which contradicts the boundedness hypothesis on $P$.
\end{proof}

\end{document}